\documentclass[runningheads,a4paper]{elsarticle}
\usepackage{amsmath,amsfonts,amsthm}
\usepackage{amssymb,amscd}
\usepackage[enableskew]{youngtab}
\usepackage{graphicx}

\def\muntz{M\"untz }
\def\divdif{\mathord\kern.43em\vrule width.6pt height5.6pt depth-.28pt \kern-.43em\Delta}

\newtheorem{theorem}{Theorem}
\newtheorem{lemma}{Lemma}
\newtheorem{proposition}{Proposition}
\newtheorem{definition}{Definition}
\newtheorem{corollary}{Corollary}
\newtheorem{example}{Example}
\newdefinition{remark}{Remark}

\begin{document}

\title{Gelfond-B\'ezier Curves}

\author[rvt,els]{Rachid Ait-Haddou\corref{cor1}}
\ead{rachid@bpe.es.osaka-u.ac.jp}
\author[focal]{Yusuke Sakane}
\author[rvt,els]{Taishin Nomura}

\cortext[cor1]{Corresponding author}

\address[rvt]{The Center of Advanced Medical Engineering and Informatics,
Osaka University, 560-8531 Osaka, Japan}
\address[focal]{Department of Pure and Applied Mathematics,
Graduate School of Information Science and Technology,
Osaka University, 560-0043 Osaka, Japan}
\address[els]{Department of Mechanical Science and Bioengineering
Graduate School of Engineering Science,
Osaka University, 560-8531 Osaka, Japan}

\begin{abstract}
We show that the generalized Bernstein bases in \muntz spaces
defined by Hirschman and Widder \cite{hirschman} and extended 
by Gelfond \cite{Gelfond} can be obtained as limits of the 
Chebyshev-Bernstein bases in \muntz spaces with respect to an
interval $[a,1]$ as $a$ converges to zero. Such a realization
allows for concepts of curve design such as de Casteljau algorithm,
blossom, dimension elevation to be translated from the general 
theory of Chebyshev blossom in \muntz spaces to these generalized 
Bernstein bases that we termed here as Gelfond-Bernstein bases. 
The advantage of working with Gelfond-Bernstein bases lies 
in the simplicity of the obtained concepts and algorithms 
as compared to their Chebyshev-Bernstein bases counterparts.      
\end{abstract}      

\maketitle
\section{Introduction}
This work was motivated by the following rather surprising observation : 
Let $r_1,...,r_n$ be $n$ real numbers such that $0< r_1 < r_2 < 
... < r_n$. Then, for any interval $[a,b]$ such that $0 < a < b$,  
the linear \muntz space 
\begin{equation}\label{muntzspace}
E = span(1,t^{r_1},t^{r_2},...,t^{r_n})
\end{equation} 
possesses a particular basis $(B_0,B_1,...,B_n)$ called 
the Chebyshev-Bernstein basis with respect to the interval $[a,b]$
and can be characterized by the following two properties \cite{mazure1}: 
For any $t \in [a,b]$, we have
\begin{equation}\label{normalization}
\sum_{k=0}^{n} B_k(t) = 1
\end{equation}
and for any $k=0,...,n$, the function $B_k$ has a zero of 
order $k$ at $a$ and a zero of order $(n-k)$ at $b$. 
The \muntz space $E$ also possesses a different basis,
called generalized Bernstein basis, that were first defined 
by Hirschman and Widder \cite{hirschman}, extended by Gelfond 
\cite{Gelfond} and popularized in Lorentz's book \cite{lorentz}.
Due to the fact that there is a variety of bases in the 
literature that are also termed generalized Bernstein 
polynomials or bases \cite{boyer,winkel} and because the account 
given in Lorentz's book for these generalized Bernstein bases
follows more the approach taken by Gelfond than the one taken by 
Hirschman and Widder, we call these bases here, the Gelfond-Bernstein bases.
The Gelfond-Bernstein bases are in some sense a generalization of the 
classical Bernstein base over the interval $[0,1]$ of the linear 
space of polynomials. The Chebyshev-Bernstein bases are defined 
only with respect to intervals $[a,b]$ such that $a > 0$.
Our observation is the fact that when $b=1$ and $a$ converges to zero,
the Chebyshev-Bernstein bases over the interval $[a,b]$ coincide with 
the Gelfond-Bernstein bases.
To understand the peculiarity and then the consequences of this result, 
we should first recall the historical reasons for defining 
the Gelfond-Bernstein bases. In 1912, Bernstein found an ingenious 
method of proving the Weierstrass approximation Theorem, by defining 
what we now know as the Bernstein basis of the linear space
of polynomials \cite{bernstein}. In 1914, \muntz, answering a conjecture
of Bernstein, generalized the Weierstrass Theorem in the following sense 
\cite{muntz}: Given a sequence of positive real numbers 
$r_1 < r_2 < ...< r_n< ..$ such that $\lim_{n\to\infty} r_n = \infty$,
then the linear space $E=span(1,t^{r_1},t^{r_2},...,t^{r_n},...)$ 
is a dense subset of the space $C\left([0,1]\right)$ of continuous functions over 
the interval $[0,1]$ endowed with the uniform norm if and only if 
\begin{equation} \label{muntzcondition}
\sum_{i=1}^{\infty} \frac{1}{r_i} = \infty.
\end{equation}
The proof given by \muntz of the if part of the Theorem
involved rather complicated techniques on summation of Fourier series.
It was then an interesting and rather difficult problem of whether there 
exists a suitable generalization of Bernstein polynomials that could
lead to a new proof of the if part of \muntz Theorem in a fashion similar 
to Bernstein proof of the Weierstrass Theorem. 
Such generalized Bernstein bases were found by Hirschman and Widder 
in 1949. Their proof of the if part of \muntz theorem
was modified and generalized by Gelfond in 1958. These generalized 
Bernstein bases (Gelfond-Bernstein bases) were defined to specifically
handle the problem of density of \muntz spaces as a subset of the space
of continuous functions over the interval $[0,1]$ (or the interval 
$[0,b]$, $b>0$ through a change of variable). The only hints that 
these generalized Bernstein bases were the most suitable one are 
the fact that they satisfy (\ref{normalization}), they are non-negative
in the interval $[0,1]$ and most importantly that they achieve
the right generalization for proving \muntz Theorem. Now, 
coming to a more recent history, Pottmann, in 1993, defined the notion of 
Chebyshev blossom associated with any linear space 
$F = span \left(1,\phi_1,\phi_2, ...,\phi_n \right)$ such that 
$span \left( \phi_1', \phi_2', ..., \phi_n' \right)$ is an extended 
Chebyshev space of order $n$ on an interval \cite{pottmann}.
Chebyshev blossoming allows for a natural definition of 
the notion of Chebyshev-Bernstein basis associated with the linear 
space $F$ and which reveal striking similarities with the classical notions
associated with the Bernstein-B\'ezier framework such as the notions of 
control points, de Casteljau algorithm, subdivision schemes, dimension 
elevation. In the case of the \muntz space $E$ in (\ref{muntzspace}), 
we can define the notion of Chebyshev-Bernstein basis only on interval 
$[a,b]$ such that $a >0$. Therefore, a way to define a notion of 
Chebyshev-Bernstein basis of the space $E$ over the interval $[0,1]$ 
is to hope that taking the limit of Chebyshev-Bernstein basis on the 
interval $[a,1]$ with $a >0$ as $a$ converges to zero leads to meaningful
expressions that constitute a basis of the space $E$.
Our observation is that in doing so, we did not only defined 
the ``Chebyshev-Bernstein basis'' over the interval $[0,1]$,
but we also discover that they coincide with the Gelfond-Bernstein
basis. This result reflects, first of all, the ingenuity of Hirschmann, 
Widder and Gelfond in defining the right generalized Bernstein bases with 
little knowledge at the time of the most natural criteria for such a 
generalization. Furthermore, this result legitimates the use of Gelfond-Bernstein
bases in computer aided geometric design and in which the CAGD concepts can be
translated from the Chebyshev-Bernstein bases to Gelfond-Bernstein bases by a
limiting process. As we will exhibit in this work, several useful properties
of the Gelfond-Bernstein bases could be simply proven without resort to 
the limiting process. However, the notion of blossom and the derivation of 
the de Casteljau algorithm are not obvious from the classical definition of the 
Gelfond-Bernstein bases and should be derived from the limiting process.
Including the point zero in the interval under consideration
through the limiting process will have an effect of collapsing difficult 
expressions in the theory of Chebyshev blossoms in \muntz spaces to highly
simpler ones. Such simplifications are achieved through a splitting concept in the 
theory of Schur functions. This provides the theory of Gelfond-Bernstein bases
with simpler algorithms as compared to their Chebyshev-Bernstein bases
counterparts.            
The paper is organized as follows. In section 2, we recall some basic properties 
of Schur functions. In section 3, we recall our main results in \cite{aithaddou1}
regarding Chebyshev blossoming in \muntz spaces and in which the Chebyshev blossom 
and the Chebyshev-Bernstein bases are expressed in terms of Schur functions.
The definition of the Gelfond-Bernstein bases, as well as the proof that 
they coincide with the Chebyshev-Bernstein bases through a limiting process will 
be given in section 4. In section 5, we study the notion of Gelfond-B\'ezier curves, 
thereby demonstrating their adequacy to be incorporated into CAGD tools. The 
expression of Chebyshev-Bernstein bases in \muntz spaces are given in terms of Schur 
functions, while the definition of Gelfond-Bernstein bases involves divided 
differences. The connection between the two bases leads to a simple expression of 
the divided differences in terms of Schur functions. We will exhibit the usefulness of 
such expression by providing the Gelfond-Bernstein bases of some specific 
\muntz spaces. In section 7, we define the blossom associated with Gelfond-B\'ezier 
curves and give a method of deriving the de Casteljau algorithm in \muntz spaces.          
In section 8, we study the concept of dimension elevation algorithms of Gelfond-B\'ezier
curves. We define the notion of shifted Gelfond-B\'ezier curves in section 9, and show
their adequacy in curve design. We conclude in Section 10.   

\section{Schur Functions}
The theory of Schur functions will play a fundamental role in this work. 
Therefore, in this section, we fix notations and review some basic concepts 
in the theory. In the case of Schur functions associated with integer partitions,
we will follow the standard Macdonald's notations \cite{Macdonald}.

A sequence $\lambda = (\lambda_1,\lambda_2,...,\lambda_n)$ of real numbers
is said to be a {\it{real partition}} if it satisfies 
\begin{equation*}
\lambda_1 > \lambda_2 - 1 > \lambda_3 - 2 > ...> \lambda_n - (n-1) > -n.
\end{equation*}
The Schur function indexed by a real partition $\lambda$ is defined as
\begin{equation}\label{schurdeterminant}
S_{\lambda}(u_1,...,u_n) = \frac{\det ( u_i^{\lambda_j + n - j} )_{
{1 \leq i, j \leq n}}}{ \prod_{1\leq i < j \leq n}(u_i - u_j)},
\end{equation}
with the convention that L'Hospital's rule is applied whenever there
are equalities among $u_1,u_2,...,u_n$. Note that if 
$\lambda = (\lambda_1,\lambda_2,...,\lambda_n)$ is a real partition, then 
$(\lambda_1,\lambda_2,...,\lambda_n,0)$ is also a real partition.
Therefore, we will adopt the convention that if the number of variables in the Schur 
function is larger than the number of components in the real partition, then we add zeros 
to the real partition. For example, we will write $S_{(\lambda_1,\lambda_2)}(u_1,u_2,u_3,u_4)$
to mean $S_{(\lambda_,\lambda_2,0,0)}(u_1,u_2,u_3,u_4)$. In the case the elements of the 
sequence $\lambda$ are positive integers, we recover the classical notion of integer
partitions and in which the associated Schur function $S_{\lambda}(u_1,...,u_n)$
is an element of the ring $\mathbb{Z}[u_1,...,u_n]$. For integer partitions, we will 
follow the following terminology and conventions. 
The total number of non-zero components, $l(\lambda)$, will be called the length
of the integer partition $\lambda$. We will always ignore the difference between
two integer partitions that differ only in the number of their trailing zeros.
The non-zero $\lambda_i$ of the partition will be called 
the {\it parts} of $\lambda$. The {\it{weight}} $|\lambda|$ of a partition $\lambda$
is defined as the sum its parts  i.e., $|\lambda| = \sum_{i=1}^{\infty} \lambda_i$.
We will find it sometimes convenient to write a partition by the common 
notation that indicate the number of times each integer appears as a part 
in the partition, for example we write the partition 
$\lambda = (4,4,4,3,3,1)$ as $\lambda = (4^{3},3^{2},1)$. 
We will adopt the convention that $S_\lambda(u_1,...,u_n) \equiv 0$
if $l(\lambda) > n$. From the definition, the Schur function associated 
with the {\it empty} partition $\lambda = (0,...,0,..)$ 
is $S_{\lambda}(u_1,...,u_n) \equiv 1$. For the partition $\lambda = (r)$,
the Schur function $S_{\lambda}$ is the complete symmetric function $h_{r}$ i.e., 
$$
S_{(r)}(u_1,u_2,...,u_n) = h_{r}(u_1,...,u_n) = 
\sum_{i_1 \leq i_2 \leq ...\leq i_r} u_{i_1} u_{i_2}...u_{i_r},
$$ 
while for the partition $\lambda = (1^r)$ with $r \leq n$, 
the Schur function $S_{(1^r)}$ is given by the elementary symmetric
function
$e_r$ i.e,
$$
S_{(1^{r})}(u_1,u_2,...,u_n) = e_{r}(u_1,...,u_n) = 
\sum_{i_1 < i_2 < ...< i_r} u_{i_1} u_{i_2}...u_{i_r}.
$$ 
The Schur function $S_{\lambda}$, with $\lambda$ an integer partition, 
can be expressed in terms of the complete symmetric functions 
through the Jacobi-Trudi formula
\begin{equation}\label{jacobi-trudi}
S_{\lambda} = \det \left( h_{\lambda_i-i+j} \right)_{1 \leq i, j \leq n},
\end{equation} 
where we assume that $h_{m}\equiv 0$ if $m < 0$.
The {\it{conjugate}}, $\lambda'$, of an integer partition
$\lambda$ is the integer partition whose Young diagram is the transpose
of the Young  diagram of $\lambda$, equivalently $\lambda'_{i} = Card \{j | \lambda_j \geq i \}.$
Using the conjugate partition, Schur functions can be expressed in terms of the 
elementary symmetric functions through the N\"agelsbach-Kostka formula
\begin{equation*}
S_{\lambda} = \det \left( e_{\lambda_i'-i+j} \right)_{1 \leq i, j \leq n},
\end{equation*} 
where we assume that $e_{m}\equiv 0$ if $m < 0$.
Throughout this work, we will use the notation 
$$
S_{\lambda}(u_1^{m_1},u_2^{m_2},...,u_k^{m_k}),
$$ 
to mean the evaluation of the Schur function in which the argument $u_1$ 
is repeated $m_1$ times, the argument $u_2$ is repeated $m_2$ times and so on. 
\vskip 0.2 cm
\noindent{\bf{Combinatorial definition of Schur functions: }}The {\it Young diagram}
of an integer partition $\lambda$ is a sequence of $l(\lambda )$
left-justified row of boxes, with the number of boxes in the $i$th row being
$\lambda_i$ for each $i$. A box $x =(i,j)$ in the diagram of $\lambda$ 
is the box in row $i$ from the top and column $j$ from the left. For example 
the Young diagram of the partition $\lambda=(5,4,2)$ and the coordinate of its boxes  
are
$$
\lambda = (5,4,2) \qquad 
\newcommand{\ff}{\mbox{\small{(1,1)}}}
\newcommand{\fs}{\mbox{\small{(1,2)}}}
\newcommand{\ft}{\mbox{\small{(1,3)}}}
\newcommand{\ffo}{\mbox{\small{(1,4)}}}
\newcommand{\ffi}{\mbox{\small{(1,5)}}}
\newcommand{\sfs}{\mbox{\small{(2,1)}}}
\newcommand{\sss}{\mbox{\small{(2,2)}}}
\newcommand{\st}{\mbox{\small{(2,3)}}}
\newcommand{\sfo}{\mbox{\small{(2,4)}}}
\newcommand{\tf}{\mbox{\small{(3,1)}}}
\newcommand{\tss}{\mbox{\small{(3,2)}}}
{\Yvcentermath1 \Yboxdim{20pt} \young(\ff\fs\ft\ffo\ffi,\sfs\sss\st\sfo,\tf\tss)} 
$$  
A {\it semi-standard tableau} $T^{\lambda}$ with entries less or equal to $n$ 
is a filling-in the boxes of the integer partition $\lambda$ with numbers 
from $\{ 1, 2, ..., n \}$ making the rows increasing when read from left to 
right and the column strictly increasing when read from the top to bottom. 
We say that the shape of $T^{\lambda}$ is $\lambda$. For each semi-standard
tableau $T^{\lambda}$ of the shape $\lambda$, we denote by $p_i$ the number
of occurrence of the number $i$ in the semi-standard tableau $T^{\lambda}$.
The weight of $T^{\lambda}$ is then defined as the monomial
$$
u^{T^{\lambda}} = u_1^{p_1} u_2^{p_2} ...u_n^{p_n}. 
$$
For a given integer partition $\lambda$ of length at most $n$, 
the Schur function $S_{\lambda}(u_1,...,u_n)$
is given by
$$
S_{\lambda}(u_1,u_2,...,u_n) = \sum_{T^{\lambda}} u^{T^{\lambda}},
$$
where the sum run over all the semi-standard tableaux of shape $\lambda$ and entries 
at most $n$.

\begin{example}
Consider the partition $\lambda = (2,1)$ and $n = 3$. Then, the
Young diagram of $\lambda$ and the complete list of semi-standard tableaux 
of shape $\lambda$ are
$$
\young(11,2) \quad \young(11,3) \quad 
\young(12,3) \quad \young(13,2) \quad 
\young(12,2) \quad \young(13,3) \quad
\young(22,3) \quad \young(23,3)
$$
Therefore, the Schur function associated with the partition $\lambda$
is given by 
$$
S_{\lambda}(u_1,u_2,u_3) = u_1^2 u_2 + u_1^2 u_3 + 2 u_1 u_2 u_3 +
u_2^2 u_3 + u_2 u_3^2 + u_1 u_2^2 + u_1 u_3^2.
$$
\end{example}
\vskip 0.2 cm
\noindent{\bf{Giambelli formula: }}
The Young diagram of an integer partition $\lambda$ is said to be a \textit{hook diagram} 
if the partition $\lambda$ is of the shape $\lambda = (p+1,1^{q})$ i.e., 
\vskip 1cm \hskip 4.277cm 
$q \left\{\Yvcentermath1 {\yng(1,1,1,1)}\right.$
\vskip -2.55 cm \hskip 4.82cm  
$\overbrace{\Yvcentermath1 {\yng(5)}}^{p+1}$ 
\vskip 1.8 cm
In Frobenius notation, we write the partition $\lambda$ as $(p|q)$.
Expanding the Jacobi-Trudi formula (\ref{jacobi-trudi}) along the top row,
shows that the Schur function associated with the partition $(p|q)$ is given by 
\begin{equation*}
S_{(p|q)} = h_{p+1} e_{q} - h_{p+2}e_{q-1} + .... + (-1)^{q} h_{p+q+1}.
\end{equation*}
Any integer partition $\lambda$ can be represented in Frobenius notation as 
\begin{equation}\label{frobenius}
\lambda = ( \alpha_1, ..., \alpha_r|\beta_1,..., \beta_r),
\end{equation}
where $r$ is the number of boxes in the main diagonal of the Young 
diagram of $\lambda$ and for $i=1,...,r$, $\alpha_i$ (resp. $\beta_i$)
is the number of boxes in the $i$th row (resp. the $i$th column) of 
$\lambda$ to the right of $(i,i)$ (resp. below $(i,i)$). For example 
the partition $\lambda = (6,4,2,1^{2})$, depicted below,
can be written in Frobenius notation as $\lambda = (5,2|4,1)$ 
\newcommand{\bbbox}{\blacksquare}
\newcommand{\nth}{\mbox{}}
$$
\lambda = {\Yvcentermath1 \young(\bbbox\nth\nth\nth\nth\nth,\nth\bbbox\nth\nth,\nth\nth,\nth,\nth)} 
$$ 
With the decomposition (\ref{frobenius}) of $\lambda$ in hook diagrams,
the Giambelli formula states that 
\begin{equation*}
S_{\lambda} = \det ( S_{\left( \alpha_i | \beta_j \right)} )_{1\leq
i,j \leq r}
\end{equation*}
We will adopt the convention that $S_{(\alpha|\beta)} \equiv 0$ if $\alpha$ or $\beta$
are negatives.
\vskip 0.2 cm
\noindent{\bf{Hook length formula: }}
The {\it {hook-length}} of an integer partition $\lambda$ at a box $x = (i,j)$ is defined
to be $h(x) = \lambda_i + \lambda'_{i} - i - j + 1$, where $\lambda'$ is the conjugate
partition of $\lambda$. In other word 
the hook-length at the box $x$ is the number of boxes that are in the same row 
to the right of it plus those boxes in the same column below it,
plus one (for the box itself). The {\it content} of the partition $\lambda$ at the box 
$x = (i,j)$ is defined as $c(x) = j-i$. The hook-length and the content of every box 
of the partition $\lambda = (5,4,2)$ is given as  
\newcommand{\negone}{\mbox{-1}}
\newcommand{\negtwo}{\mbox{-2}}
$$ 
h(\lambda) = {\Yvcentermath1 \young(76431,5421,21)} 
\qquad
Content(\lambda) = {\Yvcentermath1 \young(01234,\negone 012,\negtwo \negone)} 
$$
With these notations, the number of semi-standard 
tableaux of shape $\lambda$ with entries at most $n$ is given by 
the so-called hook-length formula as 
\begin{equation}\label{combinatorialhook}
f_{\lambda}(n) = S_{\lambda}(\overbrace{1,1,...,1}^{n}) =
\prod_{x \in \lambda}\frac{n+c(x)}{h(x)}.     
\end{equation}
In particular, we have the following useful hook-length formulas 
\begin{equation}\label{normalization1}
f_{(1^r)}(n) = \binom{n}{r}, \quad f_{(r)}(n) = \binom{n+r-1}{r} 
\end{equation}
and 
\begin{equation}\label{normalization2}
f_{(p|q)}(n) = \frac{n}{p+q+1} \binom{n+p}{p} \binom{n-1}{q}.
\end{equation}
We will adopt the convention that for every integer $n$, the hook-length 
of the empty partition $\lambda = (0,0,...)$ is given by  
$f_{\emptyset}(n) = 1$.
We can also show that for any real partition $\lambda$, we have    
\begin{equation}\label{generalhook}
f_{\lambda}(n) = \frac{\prod_{1\leq j < k \leq n}
(\lambda_{j} - \lambda_{k} -j + k)}{\prod_{j=1}^{n}(j-1)!}.
\end{equation}
\vskip 0.2 cm
\noindent{\bf{Skew Schur functions and Branching rule: }}
Given two integer partitions, $\lambda$ and $\mu$, such that $\mu \subset \lambda$ i.e.,
$\mu_i \leq \lambda_i$, $i \geq 1$, a Young diagram with skew shape $\lambda/\mu$ is the
Young diagram of $\lambda$ with the Young diagram of $\mu$ removed from its upper left-hand
corner. Note that the standard shape $\lambda$ is just the skew shape $\lambda/\mu$ 
with $\mu = \emptyset$. For example, we have 
$$
\newcommand{\nothings}{\mbox{}}
(4,3,1)/(2,1) ={\Yvcentermath1 
\young(::\nothings\nothings,:\nothings\nothings,\nothings)}
$$
The skew Schur function $S_{\lambda/\mu}$ is defined as 
$$
S_{\lambda/\mu}(u_1,u_2,...,u_n) = \sum_{T^{\lambda/\mu}} x^{T^{\lambda/\mu}}
$$
where the sum run over all the semi-standard tableaux of shape $\lambda/\mu$
and entries at most $n$. Skew Schur functions have a determinant expression as
$$
S_{\lambda/\mu} = det(h_{\lambda_i -\mu_j - i + j})_{1\leq i,j \leq n}.
$$ 
Using the skew Schur functions, we have the following branching rule
\begin{equation*}
S_{\lambda}(u_1,...,u_{j},u_{j+1},...,u_{n}) = 
\sum_{\mu \subset \lambda} S_{\mu}(u_1,...,u_j) 
S_{\lambda/\mu}(u_{j+1},...,u_{n}).
\end{equation*} 
Particularly interesting for this work, the following two branching rules 
\begin{equation}\label{branchingrule1}
S_{\lambda}(u_1,...,u_{n-1},u_n) = 
\sum_{\mu \prec \lambda} S_{\mu}(u_1,...,u_{n-1}) u_{n}^{|\lambda|-|\mu|},
\end{equation}
where the sum is over are the interlacing partitions $\mu$ i.e., partition 
$\mu = (\mu_1,...,$ $\mu_{n-1})$ such that
\begin{equation*}
\lambda_1 \geq \mu_1 \geq \lambda_2 \geq ...\mu_{n-1} \geq \lambda_n,  
\end{equation*}
and 
\begin{equation}\label{branchingrule2}
S_{\lambda}(u_1,...,u_{n-1},u_n) = 
\sum_{j=0}^{\lambda_1} S_{\lambda/{(j)}}(u_1,...,u_{n-1}) u_{n}^{j}.
\end{equation}
\vskip 0.2 cm
\noindent{\bf{Splitting formula for Schur functions: }}
The following splitting formula for Schur functions will be fundamental 
in this work. For integer partitions, it can be proved using the 
branching rule of Schur functions. For real partitions, its proof is 
explicit in the treatment given in \cite{biane} even though such a proof
is given only for integer partitions. To be rigorous, we will repeat
the exact same proof here and only emphasize the part which makes the 
arguments of the proof valid for real partitions too.   
\begin{proposition}
Let $\eta = (\lambda_1, . . . , \lambda_k, \mu_1, . . . , \mu_h)$
be a real partition. Then we have
\begin{equation}\label{split}
\lim_{\epsilon\to 0} \frac{S_{\eta}(z_1,...,z_k,\epsilon y_1,...,\epsilon y_h)}
{\epsilon^{|\mu|}} = 
S_{\lambda}(z_1,...,z_k) S_{\mu}(y_1,...,y_h)
\end{equation}
where $\lambda$ and $\mu$ are the real partitions
$\lambda = (\lambda_1,...,\lambda_k)$ and $\mu = (\mu_1,...,\mu_h),$
where $|\mu|$ denotes $\mu_1 + \mu_2 +....+\mu_h$.
\end{proposition}
\begin{proof}
Without loss of generality, we can assume that the components of the vector 
$(z_1,...,z_k,y_1,...,y_h)$ are pairwise distinct. Consider, now, 
a generic real partition $\gamma = (\alpha_1,\alpha_2,...,\alpha_k,\beta_1,...,\beta_h)$
and let us study the behavior of the function $\Delta_{\gamma}(z,\epsilon y)$
as the real number $\epsilon$ converges to zero. The function  
$\Delta_{\gamma}(z,\epsilon y)$ is defined as the determinant of the $(n\times n)$ 
matrix $V$ defined as 
$$
V_{ij} = z_{i}^{\gamma_j+k+h-j} \quad \textnormal{for}
\quad 1 \leq i \leq k, \quad j=1,...,n
$$
and
$$
V_{ij} = (\epsilon y_{i-k})^{\gamma_j+k+h-j} \quad \textnormal{for}
\quad k < i \leq n, \quad  j=1,...,n.
$$ 
Consider the Laplacian expansion of the determinant of $V$ along the first 
$k$ rows 
\begin{equation}\label{split1}
\det V = \sum_{I \subset [k+h], |I|=k} \rho(I,[k]) 
\det V_{[k],I} \det V_{[k]^{c},I^{c}},
\end{equation}
For any set of indices $I$ with $|I| = k$, the determinant $\det V_{[k]^{c},I^{c}}$
in (\ref{split1}) has an exposed factor of $\epsilon^{\sum_{j \in I^{c}} \gamma_{j}+k+h-j}$.
In particular for $I = [k]$, the factor is given by $\epsilon^{|\beta| + \binom{h}{2}}$.
The fact that $\gamma$ is a real partition shows, in particular,
that for any $I$ such that 
$|I|=k$ and $I \neq [k]$, we have 
$$
\sum_{j \in I^{c}} \gamma_{j}+k+h-j > |\beta| + \binom{h}{2}.
$$
Therefore, we have 
\begin{equation}\label{split2}
\begin{split}
\frac{\Delta_{\gamma}(z,\epsilon y)}{\epsilon^{|\beta| + \binom{h}{2}}} & =
\det(z_{i}^{\alpha_j+k+h-j})_{1\leq i,j\leq k} 
\det(y_{i}^{\beta_j+h-j})_{1\leq i,j\leq h} + O(\epsilon^{\tau}) \\
& = \left( \prod_{i=1}^{k} {z_i}^{h}\right) \Delta_{\alpha}(z) 
\Delta_{\beta}(y) + O(\epsilon^{\tau}), 
\end{split}
\end{equation}
where $\tau$ is a strictly positive number. Applying Equation (\ref{split2})
to the real partition $\eta$ and the zero partition lead to (\ref{split}).
\end{proof}

\section{Chebyshev blossom in \muntz spaces and Chebyshev-Bernstein bases}
In this section, we review the needed results that we have obtained in
\cite{aithaddou1} on Chebyshev blossom in \muntz spaces. We recall the expression of the
Chebyshev blossom in terms of Schur functions, we give the expression of 
the pseudo-affinity factor, as well as an explicit expression of 
the Chebyshev-Bernstein bases.  
\vskip 0.2 cm 
\noindent{{\bf {Chebyshev blossom: }}}
Let $\Lambda = (r_0,r_1,...,r_{n})$ be a sequence of $(n+1)$ real
numbers such that $0 =r_0 < r_1 < ... < r_n$ and let $I = [a,b]$ be
a non-empty real interval such that $0<a<b$. The function 
\begin{equation}\label{chebyshevcurve}
\phi(t)=(t^{r_1},t^{r_2},...,t^{r_n})^{T}  
\end{equation}
is a {\it{Chebyshev function of order $n$}} on $I$ \cite{mazure0}.
Therefore, if we denote by $Osc_{i}\phi(t)$ the osculating flat of order $i$
of the function $\phi$ at the point $t$, i.e.,
$$
Osc_i\phi(t) = \{\phi(t)+\alpha_1 \phi'(t)+... + \alpha_i \phi^{(i)}(t) 
\quad | \quad 
\alpha_1,...,\alpha_i \in \mathbb{R} \}, 
$$
then, for all distinct points $\tau_1,..., \tau_r$ in the interval $I$
and all positive integers $\mu_1, ..., \mu_r$ such that 
$\sum_{k = 1}^r \mu_k = m \leq n$, we have 
\begin{equation}\label{intersection}
dim \cap_{k = 1}^r Osc_{n-\mu_k}\phi(\tau_k) = n-m.
\end{equation}
In particular, if in equation (\ref{intersection}) we have $m=n$,
then the intersection consists of a single point in $\mathbb{R}^{n}$,
which we label as 
$\varphi(\tau_1^{\mu_1},\tau_2^{\mu_2},...,\tau_r^{\mu_r})$, i.e.,
$$
\varphi(\tau_1^{\mu_1},\tau_2^{\mu_2},...,\tau_r^{\mu_r}) = 
\cap_{k = 1}^r Osc_{n - \mu_k} \varphi ( \tau_k ).
$$
The previous construction provides us with a function 
$\varphi = (\varphi_1,\varphi_2, ..., \varphi_n)^{T}$ 
from $I^n$ into $\mathbb{R}^{n}$ with the following 
straightforward properties: The function $\varphi$ is symmetric 
in its arguments and its restriction to the diagonal of $I^n$ 
is equal to $\phi$ i.e., $\varphi(t,t,...,t) = \phi(t)$. 
The function $\varphi$ is called the {\it{Chebyshev blossom}}
of the function $\phi$. To give an explicit expression of 
the Chebyshev blossom $\varphi$ of the function $\phi$, we first 
associated a real partition $\lambda$ to the sequence $\Lambda$ as 
follows

\begin{definition}\label{realpartition}
For a sequence $\Lambda = (r_0,r_1,...,r_{n})$ of $(n+1)$ real
numbers such that $0 =r_0 < r_1 < ... < r_n$, we define the real 
partition $\lambda = (\lambda_1,...,\lambda_{n})$
associated with the sequence $\Lambda$ by
\begin{equation}\label{realpartition}
\lambda_k = r_n - r_{k-1} - (n-k+1) 
\quad \textnormal{for} \quad k=1,...,n.
\end{equation}
\end{definition}

We need also to define a sequence of real partitions associated with 
a single real partition $\lambda$ as follows 

\begin{definition}
Let $\lambda = (\lambda_1,\lambda_2,...,\lambda_n)$ be a
real partition. The \muntz tableau associated with the
partition $\lambda$ is given by a sequence of $(n+1)$ real
partitions $(\lambda^{(0)},\lambda^{(1)},\lambda^{(2)},...,
\lambda^{(n)})$ defined as follows:   
$$
\lambda^{(0)} =(\lambda_2,\lambda_3,...,\lambda_n),
$$
for $\; i=1,2,...n-1$
$$
\lambda^{(i)} = (\lambda_1 + 1,\lambda_2 + 1,...,\lambda_{i}+1,
\lambda_{i+2},...,\lambda_n)
$$
and 
$$
\lambda^{(n)} = (\lambda_1+1,\lambda_2+1,...,\lambda_n+1).
$$
\end{definition}

In the case of integer partitions, a way to remember the construction 
of the \muntz tableau is to remark that the partition $\lambda^{(0)}$
is obtained form the partition $\lambda$ by deleting the first row.
The partition $\lambda^{(i)}$ is obtained by adding 
a box to the first $i$ rows of the partition $\lambda$,
deleting the $i+1$ row and keeping all the other rows the same.

For a real partition $\lambda$, the real partition 
$\lambda^{(0)}$ in the \muntz tableau associated with $\lambda$ 
will play an important role in this work and will be called 
{\it {the bottom partition of $\lambda$}}.

\begin{example}
The \muntz tableau associated with the partition $\lambda = (4,2)$ and $n =3$ 
is depicted as
$$  
\lambda = {\Yvcentermath1 \yng(4,2)} \quad 
\lambda^{(0)} = {\Yvcentermath1 \yng(2)} \quad
\lambda^{(1)} = {\Yvcentermath1 \yng(5)} \quad
\lambda^{(2)} = {\Yvcentermath1 \yng(5,3)} \quad
$$
$$
\lambda^{(3)} = {\Yvcentermath1 \yng(5,3,1)}
$$   
\end{example} 

{\bf{Notations 1: }}To a sequence $\Lambda = (0=r_0,r_1,r_2,...,r_n)$ 
of strictly increasing real numbers, we can associate the Chebyshev 
curve given in (\ref{chebyshevcurve}). We can also associated the 
\muntz space $E = span(1,t^{r_1},t^{r_2},...,t^{r_n})$. 
To emphasize the dependence of $E$ on the sequence $\Lambda$, we 
will denote this space as $E_{\Lambda}(n)$. From definition \ref{realpartition}, 
we can also associate a real partition $\lambda$ to the sequence 
$\Lambda$. Therefore, we will also denote the space $E$ as 
$\mathcal{E}_\lambda$, if we want to emphasize more the real 
partition $\lambda$ than the sequence $\Lambda$. 
In case, we want to emphasize both the sequence 
$\Lambda$ and the partition $\lambda$, we will write 
$E_{\Lambda}(n) = \mathcal{E}_{\lambda}(n)$ 
in the corresponding statement.
\vskip 0.2 cm
With the definitions aboves, the following explicit expression of 
the Chebyshev blossom of the Chebyshev curve $\phi$ given in 
(\ref{chebyshevcurve}) has been proven in \cite{aithaddou1}
\begin{theorem}\label{theoremblossom}
For any sequence $(u_1,u_2,...,u_n)\in]0,+\infty[^n$,
the blossom $\varphi = (\varphi_1,\varphi_2,...,\varphi_n)^{T}$
of the Chebyshev curve $\phi$ given in (\ref{chebyshevcurve})
is given by
$$
\varphi_i ( u_1, u_2, ..., u_n ) = \frac{f_{\lambda^{(0)}} ( n )
S_{\lambda^{(i)}}(u_1,u_2,..., u_n)}
{f_{\lambda^{(i)}}(n) S_{\lambda^{(0)}}(u_1,u_2,...,u_n)},
$$
where $(\lambda^{(0)},\lambda^{(1)},...,\lambda^{(n)})$ 
is the \muntz tableau associated with the real partition 
$\lambda$, which in turn $\lambda$ is the real partition associated with 
the sequence $\Lambda$. 
\end{theorem}
\vskip 0.2 cm 
\noindent{{\bf {The pseudo-affinity property: }}}Another fundamental property of 
Chebyshev blossom is the notion of pseudo-affinity, which states that for 
any Chebyshev curve $\phi$ on an interval $I$, there exists a function $\alpha$ such 
that for any distinct numbers $a$ and $b$ in the interval $I$, and for any $t \in I$, we have      
\begin{equation}\label{pseudoaffinity}
\varphi (u_1,...,u_{n-1},t)=\left(1-\alpha(t)\right)\varphi(u_1,...,u_{n-1},a) +
\alpha(t) \varphi(u_1,...,u_{n-1},b),
\end{equation}
where $\varphi$ is the Chebyshev blossom of the function $\phi$. 
In general, the function $\alpha$ depends on $a, b$,
the real numbers $u_i,i=1,...,n-1$ as well as the parameter $t$.
To stress this dependence, we will often write the 
pseudo-affinity factor as $\alpha(u_1,...,u_{n-1};a,b,t)$. 
In the case of the Chebyshev curve given in (\ref{chebyshevcurve}), we 
can give an explicit expression of the pseudo-affinity factor 
as follows \cite{aithaddou1}
\begin{theorem}\label{pseudotheorem}
The pseudo-affinity factor of the \muntz space $\mathcal{E}_{\lambda}(n)$
associated with a real partition $\lambda= (\lambda_1,...,\lambda_n)$ is given by 
$$
\alpha(U;a,b,t) = ( \frac{t-a}{b-a} ) \frac{S_{\lambda}(U,a,t) S_{\lambda^{(0)}}(U,b)}
{S_{\lambda}(U,a,b) S_{\lambda^{(0)}}(U,t)}, 
$$
where $U$ is a sequence of strictly positive real numbers $U = (u_1,...,u_{n-1})$ and 
$\lambda^{(0)}$ is the bottom partition of $\lambda$.
\end{theorem}
\vskip 0.2 cm 
\noindent{{\bf {Chebyshev-Bernstein Basis: }}}Given two real numbers $a$ and $b$ 
such that ($0< a < b$), and denote by $\Pi_k$, $k=0,...,n$, the $(n+1)$
points defined as 
$$
\Pi_i = \varphi(a^{n-i},b^{i}), 
$$
where $\varphi$ is the Chebyshev blossom of the Chebyshev curve $\phi$
in (\ref{chebyshevcurve}). Denote by $\Lambda$ (resp. $\lambda$) the sequence 
(resp. the real partition) associated with the curve $\phi$. The points $\Pi_i$
are affinely independent in $\mathbb{R}^{n}$ \cite{mazure1}. Therefore, there exist
$(n+1)$ functions $B^{n}_{k,\lambda}, k=0,...,n$ such that for any $t \in I$
$$
\phi(t) = \sum_{k=0}^n B^{n}_{k,\lambda}(t)\Pi_i \quad \textnormal{and} 
\quad \sum_{k=0}^n B^{n}_{k,\lambda}(t)=1.
$$
The functions $B^{n}_{0,\lambda},...,B^{n}_{k,\lambda},..., B^{n}_{n,\lambda}$
form a basis of the \muntz space $E_{\Lambda}(n)= \mathcal{E}_{\lambda}(n)$,
called the Chebyshev-Bernstein basis of the space 
$E_{\Lambda}(n) = \mathcal{E}_{\lambda}(n) $ with respect to the interval $[a,b]$.
An explicit expression of the Chebyshev-Bernstein basis is given 
by \cite{aithaddou1}
\begin{theorem}\label{bernsteintheorem}
The Chebychev-Bernstein basis
$(B^{n}_{0,\lambda},B^{n}_{1,\lambda},...,B^{n}_{n,\lambda})$
of the \muntz space associated with a real partition 
$\lambda = (\lambda_1,\lambda_2,...,\lambda_n)$ 
over an interval $[a,b]$ is given by 
\begin{equation}\label{bernstein}
B^{n}_{k,\lambda}(t) = \frac{f_{\lambda}(n+1)}{f_{\lambda^{(0)}}(n)}
B^{n}_{k}(t) 
\frac{S_{\lambda^{(0)}}(a^{n-k},b^{k}) t^{\lambda_1}
S_{\lambda}(a^{n-k},b^{k},\frac{ab}{t})} 
{S_{\lambda}(a^{n+1-k},b^{k}) S_{\lambda}(a^{n-k}, b^{k+1})},
\end{equation}
where $B^{n}_{k}$ is the classical Bernstein basis of the polynomial space
over the interval $[a,b]$ and $\lambda^{(0)}$ is the bottom partition of $\lambda$.
\end{theorem}

\section{Divided difference and Gelfond-Bernstein bases}
Let $f$ be a smooth real function defined on an interval $I$. 
For any real numbers $x_0 \leq x_1 \leq ... \leq x_n$ in the interval $I$,
the divided difference $[x_0,...,x_n]f$ of the function $f$ supported
at the point $x_i, i=0,...,n$ is recursively defined by $[x_0]f = f(x_0)$ 
and  
\begin{equation}\label{dvdefinition}
[x_0,x_1,...,x_n]f = \frac{[x_1,...,x_n]f - [x_0,x_1,...,x_{n-1}]f}{x_n - x_0}
\quad \textnormal{if} \quad n>0.
\end{equation}
If some of the $x_i$ coincide, then the divided difference $[x_0,...,x_n]f$
is defined as the limit of (\ref{dvdefinition}) when the distance of the $x_i$ 
becomes arbitrary small. A simple inductive argument shows that when the $x_i$ 
are pairwise distinct then we have 
{ \small {\begin{equation}\label{determinant}
[x_0,...,x_n]f = \sum^{n}_{i=0} \frac{f(x_i)}{\prod^{n}_{j=0, j \neq i}{(x_i - x_j)}}
= \frac{\left| \begin{array}{ccccc}
1 & x_{0} & \dots& x_{0}^{n-1} & f(x_0)  \\
1 & x_{1} & \dots& x_{1}^{n-1} & f(x_1)  \\
\dots & \dots & \dots & \dots            \\
1 & x_{n} & \dots& x_{n}^{n-1} & f(x_n) \\ 
\end{array} \right|} {V(x_0,x_1,...,x_n)},
\end{equation}}}
where $V(x_0,...,x_n)$ is the Vandermonde determinant. Note that by (\ref{determinant}) 
the divided difference $[x_0,x_1,...,x_n]f$ is symmetric in the arguments
$x_0,x_1,...x_n$.  
Consider, now, the function $f_{t}(x) = t^x$, where $t$ is viewed as 
a parameter. For a sequence $\Lambda = (0=r_0,r_1,...,r_n)$ of strictly 
increasing real numbers, the Gelfond-Bernstein basis of the \muntz space 
$\mathcal{E}_{\Lambda}(n)$ is defined as 
\begin{definition}
For a sequence $\Lambda=(0=r_0,r_1,...,r_n)$ of strictly
increasing positive real numbers, the Gelfond-Bernstein basis of the 
\muntz space $E_{\Lambda}(n)$ with respect to the interval $[0,1]$
is defined by
$$
H^{n}_{k,\Lambda}(t) = (-1)^{n-k} r_{k+1}...r_{n} [r_k,...,r_n]f_{t}
\quad \textnormal{for} \quad k=0,...,n-1
$$
and 
$$
H^{n}_{n,\Lambda}(t) = t^{r_n}.
$$
\end{definition} 
The determinant representation of the divided differences (\ref{determinant}),
shows that for $k=0,...,n-1$, the Gelfond-Bernstein basis can be expressed as 
\begin{equation}\label{lorentzdeterminant}
H^{n}_{k,\Lambda}(t) = \frac{r_{k+1}r_{k+2}...r_{n}}
{V(r_{k},r_{k+1},...,r_{n})} 
\left| \begin{array}{ccccc} 
t^{r_k} & 1 & r_{k} & \dots& r^{n-k-1}_{k}  \\
t^{r_{k+1}} & 1 & r_{k+1} & \dots& r^{n-k-1}_{k+1} \\
\dots & \dots & \dots & \dots \\
t^{r_{n}} & 1 & r_{n} & \dots& r^{n-k-1}_{n} \\ 
\end{array} \right|
\end{equation}
Formula (\ref{lorentzdeterminant}) reiterate the fact that every function 
$H^{n}_{k,\Lambda}$ is an element of the space $E_{\Lambda}(n)$.
Moreover, applying successive derivatives to the determinant 
formula (\ref{lorentzdeterminant}) shows that the function $H^{n}_{k,\Lambda}$
has a zero of order $n-k$ at $1$. Now let $a$ be a real number such that $0 < a <1$, 
and let $\lambda$ be the real partition associated with the sequence $\Lambda$  
and denote by $B^{n}_{k,\lambda}, k=0,...,n,$ the Chebyshev-Bernstein basis of the space 
$\mathcal{E}_{\lambda}(n) = E_{\Lambda}(n)$ over the interval $[a,1]$.
If we express the function $B^{n}_{0,\lambda}$ in the Gelfond-Bernstein basis
$H^{n}_{k,\Lambda}, k=0,...,n$ as 
$$
B^{n}_{0,\lambda}(t) = \sum^{n}_{k=0} a_{k} H^{n}_{k,\Lambda}(t),
$$
then, using the fact that $B^{n}_{0,\lambda}$ has a zero of order $n$ at $1$, 
shows that $a_1=a_2=...=a_n = 0$. Therefore, there exists a constant $a_0$ 
such that $B^{n}_{0,\lambda} = a_0 H^{n}_{k,\Lambda}$. Moreover, using
the fact that $B^{n}_{0,\lambda}(a) =1$, shows that the constant $a_0$ is 
given by  $a_0 = 1/H^{n}_{k,\Lambda}(a)$. Therefore, from the expression 
of the Chebyshev-Bernstein basis in Theorem \ref{bernsteintheorem}, we have       
$$
B^{n}_{0,\lambda}(t) = \frac{(1-t)^{n} S_{\lambda}(1,t^n) S_{\lambda^{(0)}}(a^n)}
{(1-a)^{n} S_{\lambda}(1,a^n) S_{\lambda^{(0)}}(t^n)} =  
\frac{H^{n}_{0,\Lambda}(t)}{H^{n}_{0,\Lambda}(a)}.
$$
The last equation shows in particular that there exists a constant $C$
such that   
\begin{equation}\label{constant}
\frac{(1-t)^{n} S_{\lambda}(1,t^n)} {S_{\lambda^{(0)}}(t^n)} = C H^{n}_{0,\Lambda}(t).
\end{equation}
From the determinant formulas (\ref{lorentzdeterminant}), we can readily show 
that $H^{n}_{0,\Lambda}(0) = 1$. Moreover, using the splitting formula 
(\ref{split}) for Schur functions, we obtain 
$$
\lim_{t\to0} \frac{S_{\lambda}(1,t^n)}{t^{|\lambda^{(0)}|}} = 
S_{\lambda_1}(1) S_{\lambda^{(0)}}(1^n)
\quad \textnormal{and} \quad 
\lim_{t\to0} \frac{S_{\lambda^{(0)}}(t^n)}{t^{|\lambda^{(0)}|}} = 
S_{\lambda^{(0)}}(1^n).
$$   
Thus, evaluating the left hand side factor of (\ref{constant}) 
at $t=0$ gives the value $1$. Therefore, the constant $C$ in (\ref{constant}) 
is equal to $1$. Summarizing,  
\begin{proposition}
Let $H^{n}_{k,\Lambda}, k=0,...,n$ be the Gelfond-Bernstein basis 
of the space $E_{\Lambda}(n) = \mathcal{E}_{\lambda}(n)$ with 
respect to the interval $[0,1]$. Then, we have
$$
H^{n}_{0,\Lambda}(t) = (-1)^{n} r_1 r_2 ... r_n [r_0,r_1,...,r_n]f_{t} = 
(1-t)^n \frac{S_{\lambda}(1,t^{n})}
{S_{\lambda^{(0)}}(t^{n})},  
$$
where $\lambda^{(0)}$ is the bottom partition of $\lambda$.
\end{proposition}
The last proposition also leads to the following interesting Schur 
representation of the divided difference of the function $f_t$ 
\begin{corollary}\label{dividedcorollary1}
Let $\Lambda = (0=r_{0},r_{1},...,r_{n})$ be a sequence of strictly
increasing real numbers and let $f_t$ be the function given by $f_t(x) = t^{x}$. 
Then, we have 
$$
[r_0,r_1,...,r_n]f_{t} = \frac{(-1)^{n}}{r_1 r_2 ... r_n}
(1-t)^n \frac{S_{\lambda}(1,t^{n})}{S_{\lambda^{(0)}}(t^{n})},
$$ 
where $\lambda$ is the real partition associated with the 
sequence $\Lambda$ and $\lambda^{(0)}$ the bottom partition of $\lambda$.
\end{corollary}

We will need the following simple lemma, in which its proof is left to the reader, 
as it can be readily proved using the determinant formulas of the divided
difference
\begin{lemma}\label{dividedlemma1}
For any real numbers $ m_0 < m_1 <... < m_s$, we have  
$$ 
[m_0,m_1,...,m_s]f_{t} = t^{m_0} 
[0,m_1-m_0,m_2-m_0,...,m_s-m_0]f_{t},
$$
where $f_t$ is the function defined by $f_t(x) = t^{x}$.
\end{lemma}
Using Lemma \ref{dividedlemma1} and corollary \ref{dividedcorollary1},
we can give the following Schur function representation of the 
Gelfond-Bernstein basis
\begin{proposition}\label{gelfondexpression}
Let $\Lambda=(0=r_0,r_1,...,r_n)$ be a sequence of strictly
increasing real numbers and denote by $\lambda = (\lambda_1,\lambda_2,...,
\lambda_n)$ the associated real partition. The Gelfond-Bernstein basis
of the space $E_{\Lambda}(n) = \mathcal{E}_{\lambda}(n)$ with respect 
to the interval $[0,1]$ is given, for $k \leq n-1$, by
\begin{equation*}
H^{n}_{k,\Lambda}(t) = \frac{\prod_{i=k+1}^{n}{r_{i}}}
{\prod_{i=k+1}^{n}{(r_{i}-r_{k})}} 
t^{r_k} (1-t)^{n-k} 
\frac{S_{(\lambda_{k+1},...,\lambda_{n})}(1,t^{n-k})}
{S_{(\lambda_{k+2},...,\lambda_{n})}(t^{n-k})} 
\end{equation*}
and 
$$
H^{n}_{n,\Lambda}(t) = t^{r_n}.
$$
\end{proposition}
\begin{proof}
According to Lemma \ref{dividedlemma1}, for $k \leq n-1$ we have 
\begin{equation*}
\begin{split}
H^{n}_{k,\Lambda}(t) = & (-1)^{n-k} r_{k+1}r_{k+2}...r_{n} 
[r_k,...,r_n]f_{t} \\
= & (-1)^{n-k} r_{k+1}r_{k+2}...r_{n} t^{r_k} [0,r_{k+1}-
r_{k},...,r_{n}-r_{k}]f_{t}. \\
\end{split}
\end{equation*}
Now, from corollary \ref{dividedcorollary1}, we have 
$$
H^{n}_{k,\Lambda}(t) = \frac{r_{k+1}r_{k+2}...r_{n}}
{(r_{k+1}-r_{k})(r_{k+2}-r_{k})...(r_{n}-r_{k})} 
t^{r_{k}} (1-t)^{n-k} \frac{S_{\eta}(1,t^{n-k})}
{S_{\eta^{(0)}}(t^{n-k})},
$$
where $\eta$ is the real partition associated with the sequence 
$(0,r_{k+1}-r_{k},...,r_{n}-r_{k})$.
From (\ref{realpartition}), we have $r_k = \lambda_1 - \lambda_{k+1} + k$. Therefore, 
the partition $\eta$ is given by $\eta = (\lambda_{k+1},...,\lambda_{n})$. This leads 
to the formula of the proposition.
\end{proof}

Now, we are in a position to show the relation between the 
Chebyshev-Bernstein bases and the Gelfond-Bernstein bases
in \muntz spaces 
\begin{theorem}\label{maintheorem}
Let $\Lambda=(0=r_0,r_1,...,r_n)$ be a sequence of strictly
increasing real numbers and $\lambda$ its associated real 
partition. Let $B^{n}_{k,\lambda},$ $k=0,...,n$ be the Chebyshev-Bernstein
basis associated with the partition $\lambda$ over an interval $[a,1]$ 
and let $H^{n}_{k,\Lambda}, k=0,...,n$ be the Gelfond-Bernstein basis 
associated with the same partition $\lambda$ over the interval $[0,1]$. 
Then, for any $k=0,...,n$, we have 
$$
\lim_{a\to 0} B^{n}_{k,\lambda}(t) = H^{n}_{k,\Lambda}(t).
$$
\end{theorem}
\begin{proof}
Let us fix $k$ such that $k \leq n-1$. From the splitting formula (\ref{split})
of Schur functions, we have
$$
\lim_{a\to 0} \frac{S_{\lambda^{(0)}}(1^{k},a^{n-k})}
{a^{\lambda_{k+2} + \lambda_{k+3}+ ...+\lambda_{n}}} = 
S_{(\lambda_2,...,\lambda_{k+1})}(1^{k}) S_{(\lambda_{k+2},...,\lambda_{n})}(1^{n-k}),
$$

$$
\lim_{a\to 0} \frac{S_{\lambda}(1^{k},a^{n-k},a/t)}
{a^{\lambda_{k+1} + \lambda_{k+2}+...+ \lambda_{n}}} = 
S_{(\lambda_1,...,\lambda_{k})}(1^{k}) S_{(\lambda_{k+1},...,\lambda_{n})}(1^{n-k},1/t),
$$

$$
\lim_{a\to 0} \frac{S_{\lambda}(1^{k},a^{n+1-k})}
{a^{\lambda_{k+1} + \lambda_{k+2}+...+ \lambda_{n}}} = 
S_{(\lambda_1,...,\lambda_{k})}(1^{k}) S_{(\lambda_{k+1},...,\lambda_{n})}(1^{n+1-k}),
$$
and 
$$
\lim_{a\to 0} \frac{S_{\lambda}(1^{k+1},a^{n-k})}
{a^{\lambda_{k+2} + \lambda_{k+3}+...+ \lambda_{n}}} = 
S_{(\lambda_1,...,\lambda_{k+1})}(1^{k+1}) S_{(\lambda_{k+2},...,\lambda_{n})}(1^{n-k}).
$$
Therefore, from the explicit expression of the Chebyshev-Bernstein basis in Theorem 
\ref{bernsteintheorem}, we have 
{\small {\begin{equation*}
\lim_{a\to 0} B^{n}_{k,\lambda}(t) = \binom{n}{k} t^{k+\lambda_1} (1-t)^{n-k}
\frac{f_{\lambda}(n+1)}{f_{\lambda^{(0)}}(n)}
\frac{ f_{\mu^{(0)}}(k)}{f_{\mu}(k+1)f_{\eta}(n+1-k)}
S_{\eta}(1^{n-k},1/t),
\end{equation*}}}
where $\mu = (\lambda_1,...,\lambda_{k+1})$ and 
$\eta = (\lambda_{k+1},...,\lambda_{n})$.
By the homogeneity property of Schur functions, we have 
$$
S_{\eta}(1^{n-k},1/t) = \frac{S_{\eta}(t^{n-k},1) f_{\eta^{(0)}}(n-k)}
{S_{\eta^{(0)}}(t^{n-k}) t^{\lambda_{k+1}}}.
$$
Therefore,
\begin{equation}\label{bernstein1}
\lim_{a\to 0} B^{n}_{k,\lambda}(t) = C(\lambda) 
t^{\lambda_1-\lambda_{k+1}+k}(1-t)^{n-k}
\frac{S_{\eta}(t^{n-k},1)}{S_{\eta^{(0)}}(t^{n-k})},
\end{equation}
where $C(\lambda)$ is given by
$$
C(\lambda) = \binom{n}{k} \frac{f_{\lambda}(n+1)}{f_{\lambda^{(0)}}(n)} 
\frac{ f_{\mu^{(0)}}(k) f_{\eta^{(0)}}(n-k)}{f_{\mu}(k+1) f_{\eta}(n+1-k)}.
$$
Using the hook-length formula (\ref{generalhook}),
it can then be proved that the constant $C(\lambda)$ is given by 
\begin{equation*}
C(\lambda) = \frac{r_{k+1}r_{k+2}...r_{n}}
{(r_{k+1}-r_{k})(r_{k+2}-r_{k})...(r_{n}-r_{k})}. 
\end{equation*}
Inserting the last equation into (\ref{bernstein1}) and taking into 
consideration the claim of Proposition \ref{gelfondexpression}
conclude that for $k= 0,1,...,n-1$, we have 
$$
\lim_{a\to 0} B^{n}_{k,\lambda}(t) = H^{n}_{k,\Lambda}(t).
$$
Moreover, as the Chebyshev-Bernstein and the Gelfond-Bernstein bases
are both normalized, we also get 
$$
\lim_{a\to 0} B^{n}_{0,\lambda}(t) = H^{n}_{0,\Lambda}(t).
$$
\end{proof}

\section{Gelfond-B\'ezier Curves} 
Theorem \ref{maintheorem} implies in particular that several properties of the
Chebyshev-Bernstein bases can be transfered to the Gelfond-Bernstein bases
through a limiting process. In particular, we conclude that the Gelfond-Bernstein 
basis is a non-negative normalized basis. Moreover, the fundamental variation
diminishing property, namely that for any $a \leq t_0 < t_1<...<t_n \leq 1$, 
the matrix $(B^{n}_{k,\lambda}(t_j))_{0 \leq j,k \leq n}$ 
is totally positive is also transfered to the Gelfond-Bernstein
bases through the limiting process. Summarizing,  
\begin{corollary}\label{total}
Let $\Lambda = (0=r_{0},r_{1},...,r_{n})$ be a sequence of strictly
increasing real numbers, and let $H^{n}_{k,\Lambda}, k=0,...,n$ be the 
Gelfond-Bernstein basis associated with the \muntz space $E_{\Lambda}(n)$ over 
the interval $[0,1]$, then 
\begin{equation}\label{normalize}
\sum_{k=0}^{n} H^{n}_{k,\Lambda}(t) = 1,
\quad \textnormal{and}\quad
0 \leq H^{n}_{k,\lambda}(t) \leq 1
\quad \textnormal{for all}\quad  t\in [0,1].
\end{equation}
Moreover, for any $0 \leq t_0 < t_1 < ...<t_n \leq 1$, 
the matrix  $(H^{n}_{k,\Lambda}(t_j))_{0 \leq j,k \leq n}$ 
is totally positive.   
\end{corollary}
In the case each $r_i$ in the sequence $\Lambda = (0=r_0,r_1,...,r_n)$
is a positive integer (a case in which the associated real partition 
is an integer partition), we can give the following characterization of the 
Gelfond-Bernstein basis 
\begin{theorem}
Let $\Lambda = (0=r_0,r_1,...,r_n)$ be a strictly increasing 
sequence of integers. Then the Gelfond-Bernstein basis
$H^{n}_{k,\Lambda}, k=0,...,n$ associated with the \muntz space
$E_{\Lambda}(n)$ is the unique normalized basis of $E_{\Lambda}(n)$
such that for $k=0,...,n$, $H^{n}_{k,\Lambda}$ vanish $r_k$ times 
at $0$ and $n-k$ times at $1$.
\end{theorem}
\begin{proof}
From (\ref{normalize}), we know that the Gelfond-Bernstein basis $H^{n}_{k,\Lambda}$
$k=0,...,n$ is normalized. Let $\lambda = (\lambda_1,...,\lambda_n)$ be
the integer partition associated with the sequence $\Lambda$. Then, from Proposition 
\ref{gelfondexpression}, to show that for each $k =0,..,n$ the function $
H^{n}_{k,\Lambda}$ vanish exactly $r_k$ times at $0$ and $n-k$ times at $1$,
we should, therefore, prove that for each $k \leq n-1$, the function $\psi_{k}(t)$
\begin{equation*}
\psi_{k}(t) = \frac{S_{(\lambda_{k+1},...,\lambda_{n})}(1,t^{n-k})}
{S_{(\lambda_{k+2},...,\lambda_{n})}(t^{n-k})} 
\end{equation*}
does not vanishes at $0$ and $1$. Denote by $\mu$ the partition 
$\mu = (\lambda_{k+1},...,\lambda_{n})$, then by the branching rule 
(\ref{branchingrule1}), we have 
\begin{equation*}
\psi_{k}(t) = \frac{\sum_{\eta \prec \mu} 
S_{\eta}(t^{n-k})}{S_{\mu^{(0)}}(t^{n-k})},
\end{equation*} 
where the sum is over all the interlacing partitions $\eta$ 
i.e., partitions $\eta = (\eta_1,\eta_2,...,$ $\eta_{n-k})$ such that
\begin{equation}\label{localeta}
\mu_1 \geq \eta_1 \geq \mu_2 \geq ...\eta_{n-k-1} \geq \mu_{n-k} \geq \eta_{n-k} \geq 0.
\end{equation}
Noticing that $\mu^{(0)}$ satisfies the condition (\ref{localeta}), and that
for any partition $\eta$ that satisfies (\ref{localeta}) and different from 
$\eta^{(0)}$, we have $|\eta|-|\mu^{(0)}| > 0$, we obtain 
\begin{equation*}
\psi_{k}(t) = 1 + \sum_{\eta \prec \mu; \eta \neq \mu^{(0)}}
\frac{f_{\eta}(n-k)}{f_{\mu^{(0)}}(n-k)} t^{|\eta|-|\mu^{(0)}|}. 
\end{equation*} 
Therefore, $\psi_{k}$ is a polynomial in $t$ with positive coefficients, 
and thus have no roots in the interval $[0,1]$. 
Now, to prove the uniqueness, we assume that there exist another basis 
$G^{n}_{k,\Lambda}, k=0,...,n$ of the \muntz space $E_{\Lambda}(n)$, 
that satisfies the normalization and the vanishing properties at $0$ and $1$
as mentioned in the Theorem. We will prove by induction on $k$ that 
$G^{n}_{k,\Lambda} = H^{n}_{k,\Lambda}$. For $k=0$, we can write 
\begin{equation*}
G^{n}_{0,\Lambda}(t) = \sum_{j=0}^{n} a_j H^{n}_{j,\Lambda}(t). 
\end{equation*}
Using the fact that $H^{n}_{j,\Lambda}$, for $j=0,...,n$ has a
root of order $n-j$ at $1$ and a successive evaluation of 
the $j$th derivative of $G^{n}_{0,\Lambda}$ at $1$, 
for $j=0,...,n-1$, shows that $a_1 = a_2 = ...=a_n = 0$.
Therefore, we have $G^{n}_{0,\Lambda} = a_0 H^{n}_{0,\Lambda}$. 
From the normalization and the vanishing properties for both 
$H^{n}_{k,\Lambda}$ and $G^{n}_{k,\Lambda}$, we have 
$H^{n}_{0,\Lambda}(1) = G^{n}_{0,\Lambda}(1) = 1$. Therefore, $a_0 =1$ 
and then $G^{n}_{0,\Lambda} = H^{n}_{0,\Lambda}$. Let us assume that 
$G^{n}_{j,\Lambda} = H^{n}_{j,\Lambda}$ for $j=0,...,k-1$ and then prove 
that $G^{n}_{k,\Lambda} = H^{n}_{k,\Lambda}$. We write 
\begin{equation}\label{successive}
G^{n}_{k,\Lambda}(t) = \sum_{j=0}^{n} a_j H^{n}_{j,\Lambda}(t). 
\end{equation}
Evaluating successively ${G^{n}_{k,\Lambda}}^{(h)}(1)$, $h =0,...,n-k-1$,
in the expression (\ref{successive}), shows that $a_n = a_{n-1} =...=a_{k+1} = 0$.
Similarly, evaluating ${G^{n}_{k,\Lambda}}^{(r_h)}(0)$, 
$h=0,...,k-1$ in the expression (\ref{successive}), shows that $a_0 =a_1=...=a_{k-1} = 0$.
Therefore, we have $G^{n}_{k,\Lambda} = a_k H^{n}_{k,\Lambda}$. Now, from 
the normalization condition, we have 
\begin{equation*}
\sum_{j=0}^{k}  {H^{n}_{j,\Lambda}}^{(n-k)}(1) = 0
\quad and \quad 
\sum_{j=0}^{k}  {G^{n}_{j,\Lambda}}^{(n-k)}(1) = 0.
\end{equation*}           
The induction hypothesis then shows that 
${G^{n}_{k,\Lambda}}^{(n-k)}(1)  = {H^{n}_{k,\Lambda}}^{(n-k)}(1)$,
thus $a_k = 1$ 
\end{proof} 

We can define the Gelfond-B\'ezier curve using the Gelfond-Bernstein 
basis in the same way we define the B\'ezier curve using the Bernstein 
basis, namely, 

\begin{definition}
Let $\Lambda = (0=r_{0},r_{1},...,r_{n})$ be a sequence of strictly
increasing real numbers, and let $H^{n}_{k,\Lambda}, k=0,...,n$ be the 
Gelfond-Bernstein basis associated with the \muntz 
space $E_{\Lambda}(n)$ over the interval $[0,1]$.
The parametric curve defined over $[0,1]$ by  
\begin{equation*}
P(t) = \sum_{k=0}^{n} H^{n}_{k,\Lambda}(t) P_i,
\end{equation*}
where $P_i$ are points in $\mathbb{R}^s$, $s\geq 1$, is called
a Gelfond-B\'ezier curve with control point $P_i$. The polygon 
$(P_0,P_1,...,P_n)$ is called the control polygon of the 
Gelfond-B\'ezier curve.
\end{definition}
From the definition of the Gelfond-Bernstein basis, the Gelfond-B\'ezier
curve $P$ satisfies the end conditions
\begin{equation*}
P(0) = P_0 
\quad \textnormal{and} \quad
P(1) = P_n.
\end{equation*}
Moreover, from the total positivity of the Gelfond-Bernstein basis
stated in corollary \ref{total}, the Gelfond-B\'ezier curve satisfies
the so-called variation diminishing property, namely, the number of intersection
of a hyperplane with the curve does not exceed the number of intersection
of the hyperplane with the control polygon. In the following, we will prove 
that if the sequence $\Lambda = (0=r_0,r_1,...,r_n)$ is such that $r_1$ 
is a positive integer, then the Gelfond-B\'ezier curve possesses the tangency 
property at the end points. Figure \ref{fig:figure1} shows different 
Gelfond-B\'ezier curves associated with the control points $(P_0,P_1,P_2,P_3)$
and various \muntz spaces.
\begin{figure}
\hskip 2 cm
\includegraphics[height=5.5cm]{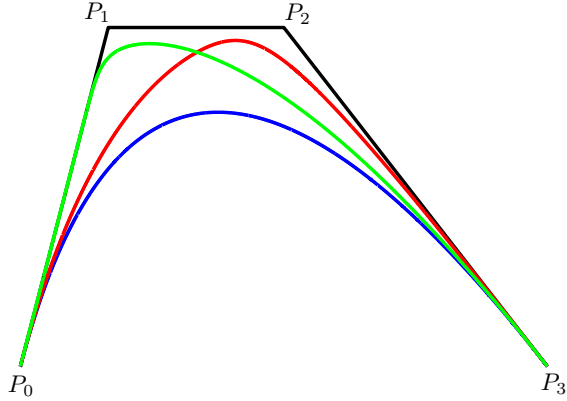}
\caption{Gelfond-B\'ezier curves associated with the control polygon 
$(P_0,P_1,P_2,P_3)$ and \muntz spaces : blue curve $span(1,t,t^2,t^3)$, 
red curve $span(1,t,t^2,t^{20})$, green curve $span(1,t^2,t^{50},t^{100})$.}
\label{fig:figure1}
\end{figure}

\vskip 0.2cm    
\noindent {\bf {The derivative of the Gelfond-Bernstein Basis: }}
We start with the following lemma giving  the derivative of the 
divided differences 
\begin{lemma}
Le $x_0,x_1,...,x_n$ be pairwise distinct real numbers  
and consider the function $\psi(t) = [x_0,x_1,...,x_n]f_t$, 
where $f_t$ is the function defined as $f_t(x) = t^{x}$.
Then we have 
\begin{equation}\label{dividedderivative1}
\psi'(t) = x_0  [x_0-1,x_1-1,...,x_n-1]f_t +  
[x_1-1,x_2-1,...,x_n-1]f_t.
\end{equation} 
\end{lemma}
\begin{proof}
By the definition of the divided differences, we have 
\begin{equation}\label{derivative10}
\psi'(t) = \sum^{n}_{i=0} \frac{x_i t^{x_i-1}}
{\prod^{n}_{j=0, j \neq i}{(x_i - x_j)}}.
\end{equation}
Now, the right hand side of the equation (\ref{dividedderivative1})
is given by 
\begin{equation}\label{dividedderivative2}
x_0 \sum^{n}_{i=0} \frac{t^{x_i-1}}
{\prod^{n}_{j=0, j \neq i}{(x_i - x_j)}}
+ \sum^{n}_{i=1} \frac{t^{x_i-1}}
{\prod^{n}_{j=1, j \neq i}{(x_i - x_j)}}.
\end{equation}
Let $k$ be an integer in $\{0,1,...,n\}$ and let us compare the coefficient 
of the monomial $t^{x_k-1}$ in both of the expressions (\ref{derivative10}) 
and (\ref{dividedderivative2}). For (\ref{derivative10}) 
the coefficient is given by 
$$
\frac{x_k}{\prod^{n}_{j=0, j \neq k}{(x_k - x_j)}},
$$
while for (\ref{dividedderivative2}), the coefficient is given for $k=0$ by 
$
\frac{x_0}{\prod^{n}_{j=0, j \neq 0}{(x_0 - x_j)}}
$ 
and for $k\neq 0$ by 
$$
\frac{1}{\prod^{n}_{j=1, j \neq k}{(x_k - x_j)}} 
\left(\frac{x_0}{x_k-x_0} +1 \right)= 
\frac{x_k}{\prod^{n}_{j=0, j \neq k}{(x_k - x_j)}}.
$$
The equality of the coefficients conclude the proof of the lemma.  
\end{proof}
A direct and simple consequence of the preceding lemma ( in which we omit the proof)
is the following
\begin{proposition}\label{derivative1}
Let $\Lambda = (0=r_{0},r_{1},...,r_{n})$ be a sequence of strictly
increasing real numbers such that $r_1 = 1$, and let
$H^{n}_{k,\Lambda}, k=0,...,n$ be the Gelfond-Bernstein 
basis, over the interval $[0,1]$, associated with the \muntz space 
$E_{\Lambda}(n)$. Then, we have
\begin{equation*}
{H_{k,\Lambda}^{n}}'(t) = \frac{\prod_{j=k+1}^{n} r_j}
{\prod_{j=k+1}^{n} (r_j-1)}
\left( r_k H^{n-1}_{k-1,\Lambda_1}(t) - (r_{k+1} -1) 
H^{n-1}_{k,\Lambda_1}(t) \right), 
\end{equation*}
where is to be understood that 
\begin{equation*}
{H_{0,\Lambda}^{n}}'(t) = - \frac{\prod_{j=2}^{n} r_j}
{\prod_{j=2}^{n} (r_j-1)}
H^{n-1}_{0,\Lambda_1}(t)
\quad \textnormal{;} \quad
{H_{n,\lambda}^{n}}'(t) = r_n H^{n}_{n,\Lambda_1}(t)
\end{equation*}
and $\Lambda_1$ is the sequence $\Lambda_1 = 
(0=r_0,r_2-1,r_3-1,...,r_n-1).$
\end{proposition}
From the last proposition, the following easily follows 
\begin{theorem}\label{r1=1}
Let $\Lambda = (0=r_{0},r_{1},...,r_{n})$ be a sequence of strictly
increasing real numbers such that $r_1 = 1$ and consider the Gelfond-B\'ezier
curve    
\begin{equation*}
P(t) = \sum_{k=0}^{n} H_{k,\Lambda}^{n}(t) P_{k}.
\end{equation*}
Then, we have 
\begin{equation*}
P'(t) =
\sum_{k=0}^{n-1} \frac{\prod_{j=k+1}^{n} r_j}
{\prod_{j=k+2}^{n} (r_j-1)}
H_{k,\Lambda_1}^{n-1}(t) \Delta{P_{k}},
\end{equation*}
where  $\Lambda_1$ is the sequence $\Lambda_1 =
(0=r_0,r_2-1,r_3-1,...,r_n-1)$.
We adopt the convention that $\prod_{k=n+1}^{n}* = 1$.
\end{theorem}
From the last theorem, we conclude that for sequences 
$\Lambda = (0=r_0,r_1,...,r_n)$ such that $r_1 = 1$, the associated
Gelfond-B\'ezier curves satisfy the tangency property 
at the end points, namely, we have  
\begin{equation*}
P'(0) =  \frac{\prod_{j=2}^{n} r_j}{\prod_{j=2}^{n} (r_j-1)}
\Delta{P_{0}} \quad \textnormal{and} \quad 
P'(1) = r_n \Delta{P_{n-1}}. 
\end{equation*}
In the case we have a sequence $\Lambda = (0=r_0,r_1,...,r_n)$ 
of strictly increasing real numbers such that $r_1 > 1$, then we can 
embed the \muntz space $E_1=span(1,t^{r_1},t^{r_2},...,t^{r_n})$ 
into the space $E_2=span(1,t,t^{r_1},t^{r_2},...,t^{r_n})$ 
in which the relation between the Gelfond-Bernstein bases of the spaces
$E_1$ and $E_2$ is given in the forthcoming proposition \ref{dimensionproposition}.
We can then apply proposition \ref{derivative1} to compute the derivatives
of the Gelfond-Bernstein bases associated with the \muntz space $E_2$.
Such a program, in which we omit the details due to their simplicity, leads to 
\begin{proposition}
Let $\Lambda = (0=r_{0},r_{1},...,r_{n})$ be a sequence of strictly
increasing real numbers such that $r_1 > 1$ and let
$H^{n}_{k,\Lambda}, k=0,...,n$ be the Gelfond-Bernstein 
basis associated with the \muntz space $E_{\Lambda}(n)$.
Then, we have, for $1 \leq k \leq n-1$
\begin{equation*}
{H_{k,\Lambda}^{n}}'(t) = \frac{\prod_{j=k+1}^{n} r_j}
{\prod_{j=k+1}^{n} (r_j-1)}
\left( r_k H^{n}_{k,\Lambda_1}(t) - (r_{k+1} -1) 
H^{n}_{k+1,\Lambda_1}(t) \right) 
\end{equation*}
and
\begin{equation*}
{H_{0,\Lambda}^{n}}'(t) = - \frac{\prod_{j=1}^{n} r_j}
{\prod_{j=2}^{n} (r_j-1)} H^{n}_{1,\Lambda_1}(t) 
\quad \textnormal{and} \quad
{H_{n,\Lambda}^{n}}'(t) = r_n H^{n}_{n,\Lambda_1}(t),
\end{equation*}
where $\Lambda_1$ is the sequence $\Lambda_1 = 
(0=r_0,r_1-1,r_2-1,...,r_n-1)$.
\end{proposition}
From the last proposition, the following easily follows 
\begin{theorem}\label{derivative2}
Let $\Lambda = (0=r_{0},r_{1},...,r_{n})$ be a sequence of strictly
increasing real numbers such that $r_1 > 1$ and consider the 
Gelfond-B\'ezier curve    
\begin{equation*}
P(t) = \sum_{k=0}^{n} H_{k,\Lambda}^{n}(t) P_{k}.
\end{equation*}
Then, we have 
\begin{equation*}
P'(t) =
\sum_{k=1}^{n} \frac{\prod_{j=k}^{n} r_j}{\prod_{j=k+1}^{n} (r_j-1)}
H_{k,\Lambda_1}^{n}(t) \Delta{P_{k-1}},
\end{equation*}
where $\Lambda_1$ is the sequence $\Lambda_1 = 
(0=r_0,r_1-1,r_2-1,...,r_n-1)$.We adopt the convention 
that $\prod_{k=n+1}^{n}* = 1$.
\end{theorem}
Note that from Theorem \ref{derivative2}, we have $P'(0) = 0$. Therefore, 
for a sequence $\Lambda = (0=r_{0},r_{1},...,r_{n})$ of strictly increasing real numbers
such that $r_1$ is a positive integer, we can iterate the statement of Theorem \ref{derivative2}
to conclude that we
have 
\begin{equation*}
P'(0) = P^{(2)}(0) = ... = P^{(r_1-1)}(0) = 0
\quad \textnormal{and} \quad 
P^{(r_1)}(0) = \frac{\prod_{j=1}^n r_j}
{\prod_{j=1}^n (r_j-r_1)} \Delta P_0, 
\end{equation*}
thereby, showing that the Gelfond-B\'ezier curve is geometrically tangent 
to the control segment $[P_0,P_1]$ for the parameter $t=0$. Moreover, we have 
$P'(1) = r_n \Delta{P_{n-1}}$ showing that Gelfond-B\'ezier curves associated 
with sequences $\Lambda = (0=r_{0},r_{1},...,r_{n})$ of strictly increasing 
real numbers such that $r_1$ is a positive integer satisfy the tangency 
property at the end points. 
  
\section{Examples of Gelfond-Bernstein bases }
For sequences $\Lambda = (0=r_0,r_1,...,r_n)$ of strictly increasing integers, 
Proposition \ref{gelfondexpression} gives an alternative method of deriving
Gelfond-Bernstein bases of \muntz spaces using the combinatoric of Schur functions
instead of computing with the divided differences. In this section,
we will exhibit the usefulness of this approach by giving the Gelfond-Bernstein 
bases of some specific \muntz spaces. As horizontal, vertical and hook Young 
diagrams occupy an important place in the combinatorics of Schur functions, 
it is only natural to define the \muntz spaces associated with these particular
Young diagrams and compute their Gelfond-Bernstein bases. We will also give one
example of a low order \muntz space for a better emphazise on our main point.
In section 2, we recalled several alternative way of computing Schur functions 
for integer partitions, such as the Jacobi-Trudi formula, the N\"agelsbach-Kostka
formula and the Giambelli formula. It is at this point of trying to derive explicit
expressions of the Gelfond-Bernstein bases using Proposition \ref{gelfondexpression} that
the reader will feel the importance of these alternative way of computing Schur 
functions and our reason of reminding them. We will not fully exhibit this fact
here, but the reader is invited to compute the Gelfond-Bernstein bases of more 
complicated \muntz spaces to be aware of the importance of the combinatorics 
of Schur functions. The same remarks apply to the computation of the blossom and
the derivation of the de Casteljau algorithms in the next section.    
\vskip 0.2cm
\noindent{\bf{Notations 2: }}In notations 1, we have denoted as $E_{\Lambda}(n)$
or $\mathcal{E}_{\lambda}(n)$ the \muntz 
space $E = span(1,t^{r_1},t^{r_2},...,t^{r_n})$ 
so as to emphasize the sequence  $\Lambda$ or the associated real 
partition $\lambda$ depending on the context.
We will imitate these notations for the Gelfond-Bernstein 
basis of the \muntz space $E_{\Lambda}(n) = \mathcal{E}_{\lambda}(n)$,
in which we denote them as $H^{n}_{k,\Lambda}$ or 
$\mathcal{H}^{n}_{k,\lambda}$ depending on the contextual emphasize. 
\vskip 0.2 cm
\vskip 0.2 cm
\noindent{\bf{Polynomial M\"untz space: }}
Consider the \muntz space associated with the sequence $\Lambda = (0,1,2,...,n)$, 
namely the \muntz space $span(1,t,t^2,...,t^n)$. 
The partition $\lambda$ associated with the sequence $\Lambda$ is the empty
partition, the bottom partition of $\lambda$ is also empty. Therefore, Theorem 
\ref{bernsteintheorem} states that the Gelfond-Bernstein basis associated with the sequence $\Lambda$ 
coincide with the classical Bernstein basis over the interval $[0,1]$.
\vskip 0.2 cm
\noindent{\bf{Combinatorial M\"untz space: }} 
Consider the \muntz space  $E_{\Lambda}(4) = span(1,t^3,$ 
$t^4,t^6,t^9)$
of order $4$. The sequence $\Lambda$ is given by  
$\Lambda = (r_0 =0,r_1 = 3,r_2 = 4,r_3 = 6,r_4 = 9)$.
The partition $\lambda$ associated with the sequence $\Lambda$ is given 
by $\lambda = (5,3,3,2)$. Let us, for example, compute the element 
$H^4_{2,\Lambda}$ of the Gelfond-Bernstein basis associated with the \muntz 
space $E_{\Lambda}(4)$. From proposition \ref{gelfondexpression}, we have 
\begin{equation*}
H^{4}_{2,\Lambda}(t) = \frac{\prod_{i=3}^{4}{r_{i}}}
{\prod_{i=3}^{4}{(r_{i}-r_{2})}} 
t^{4} (1-t)^{2}
\frac{\tiny S_{\yng(3,2)}(1,t^2)}{\tiny S_{\yng(2)}(t^2)}.
\end{equation*}
Therefore, using the branching rule (\ref{branchingrule1}),
the expression of $H^{4}_{2,\Lambda}(t)$ is given by 
\begin{equation*}
\frac{27}{5} t^{4} (1-t)^{2}
\frac{\tiny t^3  f_{\yng(3,2)}(2) + t^2  ( f_{\yng(2,2)}(2) + f_{\yng(3,1)}(2))
+  t ( f_{\yng(3)}(2) + f_{\yng(2,1)}(2)) + f_{\yng(2)}(2)}
{\tiny f_{\yng(2)}(2)}.
\end{equation*}
Using the hook length formula (\ref{combinatorialhook}), we obtain  
\begin{equation*}
H^{4}_{2,\Lambda}(t) = \frac{27}{15} t^{4}(1-t)^2 (3 + 6t + 4t^2 + 2 t^3). 
\end{equation*}
\vskip 0.2 cm
\noindent {\bf {Elementary \muntz spaces}}
Let $l$ and $n$ be two positive integers such that $1 \leq l \leq n$.
Consider the \muntz space of order $n$, defined for $l \neq 1$ by
$E = span(1,t,t^2,...,t^{l-1},t^{l+1},...,t^{n+1})$ and 
$E = span(1,t^2,t^3,...,t^{n+1})$ for $l=1$.
The partition $\lambda$ associated with the \muntz space $E$ is 
given by a vertical Young diagram with $l$ boxes, i.e,
$\lambda = (1^{l})$. For this reason, we have called these
\muntz spaces in \cite{aithaddou1} the $l$th elementary \muntz spaces. 
The sequence $\Lambda$ associated with $E$ is given by 
$\Lambda=(r_0=0,r_1=1,...,r_{l-1}=l-1,r_{l}=l+1,
r_{l+1}=l+2,...,r_{n}=n+1)$.
Let us first compute the Gelfond-Bernstein basis 
$\mathcal{H}^{n}_{k,(1^l)}$ when $k \leq l-1$. In this case
we have
\begin{equation*} 
\frac{\prod_{i=k+1}^{n}{r_{i}}}
{\prod_{i=k+1}^{n}{(r_{i}-r_{k})}} = \frac{l-k}{l} \binom{n+1}{k}.
\end{equation*}
Therefore, by Proposition \ref{gelfondexpression}, for $k \leq l-1$, we have 
$$
\mathcal{H}^{n}_{k,(1^l)}(t) = \frac{l-k}{l} \binom{n+1}{k} 
t^{k} (1-t)^{n-k} \frac{e_{l-k}(1,t^{n-k})}{e_{l-k-1}(t^{n-k})}.
$$
Using the branching rule $e_{l-k}(1,t^{n-k}) = e_{l-k}(t^{n-k}) 
+ e_{l-k-1}(t^{n-k})$, we obtain 
$$ 
\mathcal{H}^{n}_{k,(1^l)}(t) = \frac{l-k}{l} \binom{n+1}{k} 
t^{k} (1-t)^{n-k} \left( 1+ \frac{n-l+1}{l-k}t \right).
$$
For the case $k \geq l$, we have 
\begin{equation*} 
\frac{\prod_{i=k+1}^{n}{r_{i}}}
{\prod_{i=k+1}^{n}{(r_{i}-r_{k})}} =  \binom{n+1}{k+1},
\end{equation*}
and then Proposition \ref{gelfondexpression}, gives 
$$
\mathcal{H}^{n}_{k,(1^l)}(t) = \binom{n+1}{k+1} t^{k+1} (1-t)^{n-k} = B^{n+1}_{k+1}(t),
$$
where $B^{n+1}_{k+1}$ is the classical Bernstein polynomials.
Summarizing
\begin{proposition}
The Gelfond-Bernstein basis of the elementary \muntz space 
$\mathcal{E}_{(1^l)}(n)$ with respect to the interval $[0,1]$
is given by
$$
\mathcal{H}^{n}_{k,(1^l)}(t) =  \frac{l-k}{l} \binom{n+1}{k} 
t^{k} (1-t)^{n-k} \left( 1+ \frac{n-l+1}{l-k}t \right)
\quad \textnormal{for} \quad k=0,...,l-1
$$
and 
$$
\mathcal{H}^{n}_{k,(1^l)}(t) = B^{n+1}_{k+1}(t)
\quad \textnormal{for} \quad k=l,...,n,
$$
where $B^{n+1}_{k+1}$ is the classical Bernstein polynomials.
\end{proposition}
\vskip 0.5 cm
\noindent {\bf {Complete \muntz spaces}}
Let $l$ be a non-negative integer and consider the \muntz space 
of order $n$, $E = span(1,t^{l+1},...,t^{l+n})$.  
The partition associated with $E$ is given by a horizontal
Young diagram with $l$ boxes, i.e., $\lambda = (l)$.
For this reason, we call the \muntz space 
$E= \mathcal{E}_{(l)}(n)$ the $l$th complete \muntz space.
The bottom partition $\lambda^{(0)}$ is an empty partition.
The sequence $\Lambda$ is given by $\Lambda = (r_0 = 0, r_1 = l+1,
...r_j = l+j,...,r_n = l+n$. For any integer $0 \leq k \leq n$,
we have
\begin{equation*}
\frac{\prod_{i=k+1}^{n}{r_{i}}}
{\prod_{i=k+1}^{n}{(r_{i}-r_{k})}} = 
\binom{n+l}{k+l} \; \textnormal{if} \; k \neq 0 
\quad \textnormal{and} \quad
\frac{\prod_{i=k+1}^{n}{r_{i}}}
{\prod_{i=k+1}^{n}{(r_{i}-r_{k})}} = 1
\; \textnormal{if} \; k = 0. 
\end{equation*} 
Therefore, by Proposition \ref{gelfondexpression}, we have 
$$
\mathcal{H}^{n}_{0,(l)}(t) = (1-t)^n h_{l}(1,t^{n})
$$
and 
$$
\mathcal{H}^{n}_{k,(l)}(t) = \binom{n+l}{k+l} t^{l+k} (1-t)^{n-k} = 
B^{n+l}_{k+l}(t)
\quad \textnormal{for} \quad k=1,...,n. 
$$
Summarizing, 
\begin{proposition}
The Gelfond-Bernstein basis of the complete \muntz space 
$\mathcal{E}_{(l)}(n)$ with respect to the interval $[0,1]$ is given by
$$
\mathcal{H}^{n}_{0,(l)}(t) = (1-t)^n \sum_{j=0}^{l} 
\binom{n+j-1}{n-1} t^{j}
$$
and
$$
\mathcal{H}^{n}_{k,(l)}(t) = B^{n+l}_{k+l}(t)
\quad \textnormal{for} \quad k=1,...,n, 
$$
where $B^{n+l}_{k+l}$ is the classical Bernstein basis.
\end{proposition}
\vskip 0.2 cm
\noindent{\bf {Hook M\"untz spaces: }}
Let $l$ and $n$ be two positive integers and let $m$ be a 
positive integer such that $0< m < n$. Consider the \muntz space
of order $n$, $E = span(1,t^{l+1},t^{l+2},...,t^{l+m},
t^{l+m+2},...,t^{l+n+1})$. The partition $\lambda$ associated
with the space $E$ is given by a $(l,m)$-hook Young diagram,
i.e., $\lambda = (l|m)$. Therefore, we call the space 
$E =\mathcal{E}_{(l|m)}(n)$ the $(l|m)$-hook \muntz space.
For $k = 0$, Proposition \ref{gelfondexpression} gives
\begin{equation*}
\mathcal{H}^{n}_{0,(l|m)}(t) = 
(1-t)^{n}
\frac{S_{(l|m)}(1,t^n)}{e_{m}(t^n)}. 
\end{equation*}
The branching rule (\ref{branchingrule2}) leads to 
\begin{equation*}
S_{(l|m)}(1,t^n) = S_{(l|m)}(t^n) + e_m(t^n) \sum_{j=1}^{l+1} h_{l+1-j}(t^n).
\end{equation*}
Therefore, 
\begin{equation}\label{forsummarize}
\frac{S_{(l|m)}(1,t^n)}{e_{m}(t^n)} = \frac{1}{f_{(1^m)}(n)} \left( 
t^{l+1} f_{(l|m)}(n) +  \sum_{j=1}^{l+1} f_{(l+1-j)}(n) t^{l+1-j} \right),
\end{equation}
in which the terms expressing the hook lengths can be computed using 
equations (\ref{normalization1}) 
and (\ref{normalization2}). 
For $k \leq m$, we have 
\begin{equation*}
\frac{\prod_{i=k+1}^{n}{r_{i}}}
{\prod_{i=k+1}^{n}{(r_{i}-r_{k})}} = 
\frac{m+1-k}{m+1+l}\binom{l+n+1}{l+k}.
\end{equation*} 
Therefore, Proposition \ref{gelfondexpression} gives
\begin{equation*}
\mathcal{H}^{n}_{k,(l|m)}(t) = \frac{m+1-k}{m+1+l}\binom{l+n+1}{l+k}
t^{l+k}(1-t)^{n-k}
\frac{e_{m-k+1}(1,t^{n-k})}{e_{m-k}(t^{n-k})}. 
\end{equation*}
Using the branching rule for the elementary symmetric functions
leads to
\begin{equation*}
\mathcal{H}^{n}_{k,(l|m)}(t) = \frac{m+1-k}{m+1+l}\binom{l+n+1}{l+k}
t^{l+k}(1-t)^{n-k}
\left( \frac{(n-m)t}{m-k+1} + 1 \right). 
\end{equation*}
For $k > m$, we have
\begin{equation*}
\frac{\prod_{i=k+1}^{n}{r_{i}}}
{\prod_{i=k+1}^{n}{(r_{i}-r_{k})}} = 
\binom{l+n+1}{l+k+1}.
\end{equation*} 
Thus, the corresponding Gelfond-Bernstein element is given 
by
\begin{equation*} 
\mathcal{H}^{n}_{k,(l|m)}(t) = \binom{l+n+1}{l+k+1}
t^{l+k+1}(1-t)^{n-k} = B^{l+n+1}_{l+k+1}(t).
\end{equation*}
Summarizing 
\begin{proposition}
The Gelfond-Bernstein basis of the hook \muntz space 
$\mathcal{E}_{(l|m)}(n)$ with respect to the interval $[0,1]$ is given by
$$
\mathcal{H}^{n}_{0,(l|m)}(t) = 
(1-t)^{n}
\frac{S_{(l|m)}(1,t^n)}{e_{m}(t^n)}, 
$$
where an explicit expression of the Schur functions can be computed using
(\ref{forsummarize}).
For $k=1,..,m$
$$
\mathcal{H}^{n}_{k,(l|m)}(t) = \frac{m+1-k}{m+1+l}\binom{l+n+1}{l+k}
t^{l+k}(1-t)^{n}
\left( \frac{(n-m)t}{m-k+1} + 1 \right) 
$$
and for $k=m+1,...,n$
$$
\mathcal{H}^{n}_{k,(l|m)}(t) = B^{l+n+1}_{l+k+1}(t),
$$
where $B^{n+l+1}_{l+k+1}$ is the classical Bernstein basis.
\end{proposition}
\section{Blossom and the de Casteljau algorithms}
As the Gelfond-Bernstein bases are limits of the Chebyshev-Bernstein 
bases in \muntz spaces, we can extend the notion of blossom to the 
Gelfond-B\'ezier curves using Theorem \ref{theoremblossom}, which in turn 
will allows us to derive the corresponding de Casteljau algorithms  

\begin{definition}
Let $\Lambda = (0=r_{0},r_{1},...,r_{n})$ be a sequence of strictly
increasing real numbers and let $\lambda = (\lambda_1,...,\lambda_n)$
be the associated real partition. Consider an element $P$ of the 
\muntz space $E_{\Lambda}(n) = \mathcal{E}_{\lambda}(n)$ written as  
$$
P(t) = \sum_{k=0}^{n} a_k t^{r_k}.
$$
Then, the blossom $f_P$ of $P$ is defined for any $0 \leq j \leq n-1$ and 
$u_{j+1},...,u_{n}$ in $]0,1]$ by
$$
f_{P}(0^{j},u_{j+1},...,u_n) = a_0 + \lim_{\epsilon\to 0} \sum_{k=0}^{n} a_k 
\frac{f_{\lambda^{(0)}}(n)}{f_{\lambda^{(k)}}(n)}\frac{S_{\lambda^{(k)}}
(\epsilon^{j},u_{j+1},...,u_n)}
{S_{\lambda^{(0)}}(\epsilon^{j},u_{j+1},...,u_n)}, 
$$
where $(\lambda^{(0)},\lambda^{(1)},...,\lambda^{(n)})$ is the \muntz tableau 
associated with the partition $\lambda$,
and $f_P(0^{n}) = a_0$.
\end{definition}
It is clear from the definition that the blossom $f_{P}$ is symmetric in its 
arguments and that for any $t \in [0,1]$, 
$f_P(t,t,...,t) = P(t)$. Moreover, if we express the function $P$ in the
Gelfond-Bernstein basis as 
$$
P(t) = \sum_{k=0}^{n} p_{k} \mathcal{H}^{n}_{k,\lambda}(t) 
$$ 
then, the values $p_{k}, k=0,...,n$ are given by 
$$
p_{k} = f_{P}(0^{n-k},1^{k}).
$$ 
Therefore, to compute the control points of the function $P$
over the interval $[0,1]$, we need only to compute the control 
points of the functions $t^{r_k}$, $k=1,...,n$. 
Such computation is given in the following
\begin{proposition}\label{controlpoints}
Let $\mathcal{H}^{n}_{k,\lambda}, k=0,,...,n$ be the Gelfond-Bernstein basis 
of the \muntz space $\mathcal{E}_{\lambda}(n) = span(1,t^{r_1},...,t^{r_n})$.
Then, we have 
\begin{equation*}
t^{r_k} = \sum_{j=k}^{n} p_{j} \mathcal{H}^{n}_{j,\lambda}(t),
\end{equation*}
where 
\begin{equation}\label{pj} 
p_{j} = (1- \frac{r_k}{r_{j+1}}) (1- \frac{r_k}{r_{j+2}}) ...
(1- \frac{r_k}{r_{n}})
\quad \textnormal{for} \quad j=k,...,n-1
\end{equation}
and 
$$
p_n =1.
$$
\end{proposition}
\begin{proof}
Let us choose $1 \leq k \leq n-1$, and denote by $p_j$ the $j$th control point 
of the function $t^{r_k}$. From the definition of the blossom, we have 
$$
p_{j} = \frac{f_{\lambda^{(0)}}(n)}{f_{\lambda^{(k)}}(n)} 
\lim_{\epsilon\to 0} \frac{S_{\lambda^{(k)}}
(1^{j},\epsilon^{n-j})}
{S_{\lambda^{(0)}}(1^{j},\epsilon^{n-j})}. 
$$
As $\lambda^{(k)} = (\lambda_1+1,\lambda_2+1,...,
\lambda_k+1,\lambda_{k+2},...,\lambda_{n},0)$
and 
$\lambda^{(0)} = (\lambda_2,\lambda_3,...,\lambda_{n},0)$, 
it is clear from the splitting formula (\ref{split}) that if $j < k$ then $p_{j} = 0$. 
In the case $j \geq k$, then again by the splitting formula (\ref{split}), we have 
$$
p_{j} = \frac{f_{\lambda^{(0)}}(n)}{f_{\lambda^{(k)}}(n)}
\frac{f_{\mu}(j)}{f_{\eta}(j)},
$$
where
$\mu$ and $\eta$ are the real partitions 
$$
\mu = (\lambda_1+1,\lambda_2+1,...,\lambda_k + 1,
\lambda_{k+2},...,\lambda_{j+1})
$$
and 
$$
\eta = (\lambda_2,\lambda_3,...,\lambda_{k},
\lambda_{k+1},...,\lambda_{j+1}).
$$
Lengthy, yet straightforward computations, using the hook length 
formula (\ref{generalhook}), shows that $p_{j}$ is given by (\ref{pj}).
The case $k=n$ is straightforward.
\end{proof}

\begin{remark}
Note that in the polynomial case $\mathcal{E}_{\emptyset}(n) 
= span(1,t,t^2,...,t^n)$, the last proposition give the familiar fact 
that the $j$th control point $p_{j}$ of the function $t^k$ is zero if $j < k$
and for $j \geq k$, we have 
$$
p_{j} = (1-\frac{k}{j+1})(1-\frac{k}{j+2})...(1-\frac{k}{n}) = 
\frac{\binom{j}{k}}{\binom{n}{k}}.
$$
\end{remark}

\begin{remark}
Proposition \ref{controlpoints} can also be proven without resorting to the 
notion of blossoming, but instead using the Cauchy residue formula as 
in \cite{lorentz}. For the seek of comparison, and of bringing up front 
a different aspect in the theory of Gelfond-Bernstein bases,
we will include the main steps of the proof here. For a sequence 
$\Lambda = (0=r_{0},r_{1},...,r_{n})$ of strictly
increasing real numbers and using the Cauchy residue formula in can be easily 
shown that the Gelfond-Bernstein basis $H^{n}_{k,\Lambda},k=0,..,n$ associated  
with the sequence $\Lambda$ can be expressed (for $0 \leq k \leq n-1$) as 
\begin{equation}\label{cauchy}
H^{n}_{k,\Lambda} (t) = (-1)^{n-k} r_{k+1}...r_{n} \frac{1}{2\pi i}
\int_{\Gamma} \frac{t^{z}}{(z-r_k)(z-r_{k+1})...(z-r_{n})} dz,
\end{equation}  
where $\Gamma$ is any simple closed curve that contains the nodes $r_i,i=k,...,n$ 
in its interior $Int(\Gamma)$, and such that the function $t^{z}$ is holomorphic 
in a neighborhood of $Int(\Gamma) \cup \Gamma$. Let us fix a $k < n$, then 
it can be proven by induction on $n$ that $1/z-r_k$ can also be written as  
\begin{equation*}
\frac{1}{z-r_n} - \frac{r_n - r_k}{(z-r_{n-1})(z-r_n)}+
...+(-1)^{n-k} \frac{(r_{k+1} - r_{k})(r_{k+2} - r_{k})...(r_{n} - r_{k})}
{(z - r_{k})(z - r_{k+1})...(z - r_{n})}.
\end{equation*}
Multiplying the last equation as well as the function $1/z-r_k$ by $t^{z}/2\pi i$ 
and integrating over $\Gamma$ leads, after using equation (\ref{cauchy}),
to a new proof of Proposition \ref{controlpoints}.    
\end{remark}
 
To express the pseudo-affinity property of the blossom in the space
$\mathcal{E}_{\lambda}(n)$ over the interval $[0,1]$, we can just introduce 
the following pseudo-affinity factor

\begin{definition}
Let $\lambda = (\lambda_1,\lambda_2,...,\lambda_n)$ be a real partition, and let 
$c,d$ be two real numbers in the interval $[0,1]$ such that $c <d$.   
we define the function $\alpha$ by: for any $0 \leq j \leq n-1$ and 
$U = (u_{j+1},...,u_{n-1})$ in $]0,1]$ if $c \neq 0$
\begin{equation*}
\alpha(0^{j},U;c,d,t) = \lim_{\epsilon\to0}
\frac{t-c}{d-c} \frac{S_{\lambda}(\epsilon^{j},U,c,t) 
S_{\lambda^{(0)}}(\epsilon^{j},U,d)}
{S_{\lambda}(\epsilon^{j},U,c,d) S_{\lambda^{(0)}}(\epsilon^{j},U,t)}, 
\end{equation*}
while for $c = 0$, we define $\alpha$ as 
\begin{equation}\label{multi1}
\alpha(0^{j},U;0,d,t) = \lim_{\epsilon\to0}
\frac{t}{d} \frac{S_{\lambda}(\epsilon^{j+1},U,t) 
S_{\lambda^{(0)}}(\epsilon^{j},U,d)}
{S_{\lambda}(\epsilon^{j+1},U,d) S_{\lambda^{(0)}}(\epsilon^{j},U,t)}, 
\end{equation}
where $\lambda^{(0)}$ is the bottom partition of $\lambda$.
\end{definition}

Taking the limit in the pseudo-affinity factor in Theorem \ref{pseudotheorem}
of the Chebyshev blossom shows that the blossom of Gelfond-B\'ezier curves satisfies
the following pseudo-affinity property : If for a function $P$ in the \muntz space
$\mathcal{E}_{\lambda}(n)$, we denote by $f_P$ its blossom, then for any
$U = (u_1,...,u_{n-1})$, sequence of real numbers in $[0,1]$, we have 
\begin{equation*}
f_P(U,t) = (1-\alpha(U;c,d,t)) f_P(U,c) + \alpha(U;c,d,t). 
f_P(U,d)
\end{equation*}
In order to derive the de Casteljau algorithm associated
with the \muntz space $\mathcal{E}_{\lambda}(n)$ over 
the interval $[0,1]$, we should derive the pseudo-affinity 
factor when $c=0$ and $d=1$. According to (\ref{multi1}) , we have 
\begin{equation*}
\alpha(0^{j},U;0,1,t) = \lim_{\epsilon\to0}
t \frac{S_{\lambda}(\epsilon^{j+1},U,t) 
S_{\lambda^{(0)}}(\epsilon^{j},U,1)}
{S_{\lambda}(\epsilon^{j+1},U,1) S_{\lambda^{(0)}}(\epsilon^{j},U,t)}. 
\end{equation*}
Applying the splitting formula (\ref{split}) leads to 
\begin{proposition}\label{decasteljau1}
Let $\lambda = (\lambda_1,...,\lambda_n$) be a real partition. Then,
the pseudo-affinity factor of the space $\mathcal{E}_{\lambda}(n)$
is given, for $j \leq n-1$, by 
\begin{equation*}
\alpha(0^{j},U;0,1,t) =  t \frac{S_{\mu}(U,t) 
S_{\eta}(U,1)}{S_{\mu}(U,1) S_{\eta}(U,t)},  
\end{equation*}
where 
$\mu = (\lambda_1,\lambda_2,...,\lambda_{n-j})$ and 
$\eta = (\lambda_2,\lambda_3,...,\lambda_{n-j+1}).$
\end{proposition}
\vskip 0.5 cm
\noindent {\bf {de Casteljau algorithm for elementary \muntz spaces:}}
In the following, we derive the de Casteljau algorithm 
for the elementary \muntz spaces. We first study two 
special cases, namely, the \muntz space associated 
with the partition $\lambda = (1^n)$, i.e, the space 
$\mathcal{E}_{(1^n)}(n) = span(1,t,t^2,...,t^{n-1},t^{n+1})$ and 
the \muntz space associated with the partition $\lambda = (1)$, i.e.; 
the space $span(1,t^2,...,t^{n},t^{n+1})$.  
For the space $\mathcal{E}_{(1^n)}$ and according to
Proposition \ref{decasteljau1}, the pseudo-affinity factor is given,
for $j\neq 0$, by 
\begin{equation*}
\alpha(0^{j},U;0,1,t) = t \frac{e_{n-j}(U,t) e_{n-j}(U,1)}
{e_{n-j}(U,1) e_{n-j}(U,t)} = t,  
\end{equation*}
while for $j=0$, we have 
\begin{equation*}
\alpha(U;0,1,t) = t \frac{e_{n}(U,t) e_{n-1}(U,1)}
{e_{n}(U,1) e_{n-1}(U,t)} = t^2 \frac{e_{n-1}(U,1)}
{e_{n-1}(U,t)}.   
\end{equation*}
The last equations lead to the following de Casteljau algorithm
for $\mathcal{E}_{(1^n)}(n)$, in which for simplicity we exhibit the case 
of a Gelfond-B\'ezier curve of order $3$ with control points $(p_0,p_1,p_2,p_3)$
over the interval $[0,1]$ (Figure \ref{fig:figure2}), as follows 
\vskip 0.2 cm
\setlength{\unitlength}{1mm}
\begin{picture}(110,40)
\put(0, 33){\mbox{$p_0 = f_P(0,0,0)$}} \put(30,33){{\mbox{$p_1 = f_P(0,0,1)$}}}
\put(60,33){{\mbox{$p_2= f_P(0,1,1)$}}}   \put(90,33){{\mbox{$p_3 = f_P(1,1,1)$}}} 
\put(10,31){\vector(1,-1){7}} \put(40,31){\vector(-1,-1){7}}
\put(41,31){\vector(1,-1){7}} \put(70,31){\vector(-1,-1){7}}
\put(72,31){\vector(1,-1){7}} \put(101,31){\vector(-1,-1){7}}
\put(17.5, 22){\mbox{$f_P(0,0,t)$}} \put(48, 22){\mbox{$f_P(0,1,t)$}}
\put(79, 22){\mbox{$f_P(1,1,t)$}} 
\put(26,20){\vector(1,-1){7}} \put(56,20){\vector(-1,-1){7}} 

\put(57,20){\vector(1,-1){7}}  \put(86,20){\vector(-1,-1){7}} 
\put(33.5, 12){\mbox{$f_P(0,t,t)$}} \put(64.5, 12){\mbox{$f_P(1,t,t)$}} 
\put(40,10){\vector(1,-1){7}} \put(71,10){\vector(-1,-1){7}} 
\put(48.5, 2){\mbox{$f_P(t,t,t)$}}
\put(13,28){\mbox{$1-t$}}  \put(35,28){\mbox{$t$}}
\put(45,28){\mbox{$1-t$}}  \put(65,28){\mbox{$t$}}
\put(75,28){\mbox{$1-x_1$}}  \put(94.3,28){\mbox{$x_1$}}

\put(30,16){\mbox{$1-t$}}  \put(50,16){\mbox{$t$}}
\put(62,16){\mbox{$1-x_2$}}  \put(78.5,16){\mbox{$x_2$}}
\put(45,6){\mbox{$1-x_3$}}  \put(62,6){\mbox{$x_3$}}
\end{picture}
\vskip 0.2 cm
\noindent where $x_i$, $i=1,2,3$ are given by 
\begin{equation*}
x_i = t^2 \frac{e_2(t^{i-1},1^{3-i+1})}{e_2(t^{i},1^{3-i})}.
\end{equation*}
In the general case, the de Casteljau algorithm of the \muntz space 
$\mathcal{E}_{(1^n)}(n)$ is given by:

\vskip 0.3 cm

\hskip 0 cm {\bf{Given}} $p_i^0 = p_i$, $i=0,...,n$

\hskip 0 cm {\bf{for}} $r=1:n$  {\bf{do}}

\hskip 1 cm {\bf{for}} $i=0:n-r-1$ {\bf{do}} 

\hskip 1.5 cm $p_i^{r} = (1-t) p_i^{r-1} + t p_{i+1}^{r-1}$

$$ \hskip -4.3 cm x_r = t^2 \frac{e_{n-1}(t^{r-1},1^{n-r+1})}{e_{n-1}(t^{r},1^{n-r})}$$
            
\hskip 1.5 cm $p_{n-r}^{r} = (1-x_{r}) p_{n-r}^{r-1} + x_{r} p_{n-r+1}^{r-1}$

\hskip 1 cm {\bf{return}}

\hskip 0 cm {\bf{return}}

\hskip 0 cm $P(t) = p_0^n$.

\begin{figure}
\hskip 2 cm
\includegraphics[height=6cm]{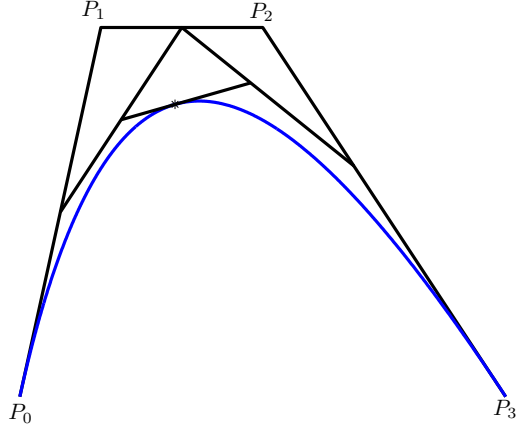}
\caption{The de Casteljau algorithm for the \muntz space 
$\mathcal{E}_{(1^3)}(3) = span(1,t,t^2,t^4)$ applied to the
Gelfond-B\'ezier curves associated with the control polygon
$(P_0,P_1,P_2,P_3)$ for the parameter $t = 1/2$.}
\label{fig:figure2}
\end{figure}

\begin{remark}
The phenomena that at each level of the de Casteljau algorithm only the edges 
of the last triangle has weights that are different from the classical de 
Casteljau algorithm is not specific to this case but the same phenomena
appears for all \muntz spaces associated with partitions of the shape
$\lambda = (r^n)$ where $r$ is a real number, namely,
\muntz spaces $span(1,t,t^2,...,t^{n-1},t^s)$ where $s$ is a real number 
strictly larger than $n-1$.        
\end{remark}
Consider, now, the pseudo-affinity factor associated with the \muntz space 
$\mathcal{E}_{(1)}(n)$. According to proposition \ref{decasteljau1}, we have 
\begin{equation*}
\alpha(0^{j},U;0,1,t) = t \frac{e_{1}(U,t)}{e_{1}(U,1)}.  
\end{equation*}
The last equation leads to the following de Casteljau algorithm
for $\mathcal{E}_{(1)}(n)$, in which again for simplicity we exhibit the case 
of a Gelfond-B\'ezier curve of order $3$ with control points $(p_0,p_1,p_2,p_3)$
over the interval $[0,1]$ (Figure \ref{fig:figure3}), as follows 
\vskip 0.2 cm
\setlength{\unitlength}{1mm}
\begin{picture}(110,40)
\put(0, 33){\mbox{$p_0 = f_P(0,0,0)$}} \put(30,33){{\mbox{$p_1 = f_P(0,0,1)$}}}
\put(60,33){{\mbox{$p_2= f_P(0,1,1)$}}}   \put(90,33){{\mbox{$p_3 = f_P(1,1,1)$}}} 
\put(10,31){\vector(1,-1){7}} \put(40,31){\vector(-1,-1){7}}
\put(41,31){\vector(1,-1){7}} \put(70,31){\vector(-1,-1){7}}
\put(72,31){\vector(1,-1){7}} \put(101,31){\vector(-1,-1){7}}
\put(17.5, 22){\mbox{$f_P(0,0,t)$}} \put(48, 22){\mbox{$f_P(0,1,t)$}}
\put(79, 22){\mbox{$f_P(1,1,t)$}} 
\put(26,20){\vector(1,-1){7}} \put(56,20){\vector(-1,-1){7}} 

\put(57,20){\vector(1,-1){7}}  \put(86,20){\vector(-1,-1){7}} 
\put(33.5, 12){\mbox{$f_P(0,t,t)$}} \put(64.5, 12){\mbox{$f_P(1,t,t)$}} 
\put(40,10){\vector(1,-1){7}} \put(71,10){\vector(-1,-1){7}} 
\put(48.5, 2){\mbox{$f_P(t,t,t)$}}
\put(13,28){\mbox{}}  \put(35,28){\mbox{$t^2$}}
\put(45,28){\mbox{}}  \put(60,28){\mbox{$\frac{t(1+t)}{2}$}}
\put(75,28){\mbox{$$}}  \put(90,28){\mbox{$\frac{t(2+t)}{3}$}}

\put(30,16){\mbox{$$}}  \put(46,16){\mbox{$\frac{2t^2}{1+t}$}}
\put(62,16){\mbox{$$}}  \put(72,16){\mbox{$\frac{t(1+2t)}{2+t}$}}
\put(45,6){\mbox{$$}}  \put(60,6){\mbox{$\frac{3t^2}{2t+1}$}}
\end{picture}
\vskip 0.2 cm
\noindent In the general case the de Casteljau algorithm of the \muntz space 
$\mathcal{E}_{(1)}(n)$ is given by :
\vskip 0.3 cm

\hskip 0 cm {\bf{Given}} $p_i^0 = p_i$, $i=0,...,n$

\hskip 0 cm {\bf{for}} $r=1:n$  {\bf{do}}

\hskip 1 cm {\bf{for}} $i=0:n-r$ {\bf{do}} 

$$
\hskip -0.26 cm 
p_i^{r} = (1-\frac{r t^2 + i t}{(r-1)t + (i+1)}) p_i^{r-1} + 
\frac{r t^2 + i t}{(r-1)t + (i+1)} p_{i+1}^{r-1}
$$

\hskip 1 cm {\bf{return}}

\hskip 0 cm {\bf{return}}

\hskip 0 cm $P(t) = p_0^n$.

In the general case of elementary \muntz space $\mathcal{E}_{(1^r)}(n)$, 
the pseudo-affinity factor is given, for $j \geq n-r+1$, by
\begin{equation*}
\alpha(0^{j},U;0,1,t) = t \frac{e_{n-j}(U,t) e_{n-j}(U,1)}
{e_{n-j}(U,1) e_{n-j}(U,t)} = t,  
\end{equation*}
while for $j < n-r+1$, we have 
\begin{equation*}
\alpha(U;0,1,t) = t \frac{e_{r}(U,t) e_{r-1}(U,1)}
{e_{r}(U,1) e_{r-1}(U,t)}.   
\end{equation*}   
We leave it as an exercise, to the reader, to derive the de Casteljau 
algorithm of the $r$th elementary \muntz space from the last equations.  
\begin{figure}
\hskip 2 cm
\includegraphics[height=6cm]{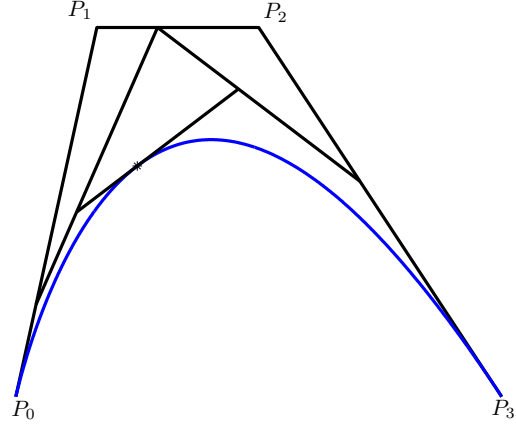}
\caption{The de Casteljau algorithm for the \muntz space 
$\mathcal{E}_{(1)}(3) = span(1,t^2,t^3,t^4)$ applied to the
Gelfond-B\'ezier curves associated with the control polygon
$(P_0,P_1,P_2,P_3)$ for the parameter $t = 1/2$.}
\label{fig:figure3}
\end{figure}

\vskip 0.5 cm
\noindent {\bf {de Casteljau algorithm for complete \muntz spaces: }}
From proposition \ref{decasteljau1}, the pseudo-affinity factor 
of the $k$th complete \muntz space is given by 
\begin{equation*}
\alpha(0^{j},U;0,1,t) = t \frac{h_{k}(U,t)}{h_{k}(U,1)}.  
\end{equation*} 
The last equation leads to the following de Casteljau algorithm
for $\mathcal{E}_{(k)}(n)$, in which for simplicity we exhibit the case 
of a Gelfond-B\'ezier curve of order $3$ with control points $(p_0,p_1,p_2,p_3)$
over the interval $[0,1]$ (Figure \ref{fig:figure4}), as follows 
\vskip 0.2 cm
\setlength{\unitlength}{1mm}
\begin{picture}(110,40)
\put(0, 33){\mbox{$p_0 = f_P(0,0,0)$}} \put(30,33){{\mbox{$p_1 = f_P(0,0,1)$}}}
\put(60,33){{\mbox{$p_2= f_P(0,1,1)$}}}   \put(90,33){{\mbox{$p_3 = f_P(1,1,1)$}}} 
\put(10,31){\vector(1,-1){7}} \put(40,31){\vector(-1,-1){7}}
\put(41,31){\vector(1,-1){7}} \put(70,31){\vector(-1,-1){7}}
\put(72,31){\vector(1,-1){7}} \put(101,31){\vector(-1,-1){7}}
\put(17.5, 22){\mbox{$f_P(0,0,t)$}} \put(48, 22){\mbox{$f_P(0,1,t)$}}
\put(79, 22){\mbox{$f_P(1,1,t)$}} 
\put(26,20){\vector(1,-1){7}} \put(56,20){\vector(-1,-1){7}} 

\put(57,20){\vector(1,-1){7}}  \put(86,20){\vector(-1,-1){7}} 
\put(33.5, 12){\mbox{$f_P(0,t,t)$}} \put(64.5, 12){\mbox{$f_P(1,t,t)$}} 
\put(40,10){\vector(1,-1){7}} \put(71,10){\vector(-1,-1){7}} 
\put(48.5, 2){\mbox{$f_P(t,t,t)$}}
\put(13,28){\mbox{$$}}  \put(29,28){\mbox{$\frac{th_r(t)}{h_r(1)}$}}
\put(45,28){\mbox{$$}}  \put(56,28){\mbox{$\frac{th_r(1,t)}{h_r(1,1)}$}}
\put(75,28){\mbox{$$}}  \put(86,28){\mbox{$\frac{th_r(1,1,t)}{h_r(1,1,1)}$}}

\put(30,16){\mbox{$$}}  \put(44,16){\mbox{$\frac{th_r(t,t)}{h_r(1,t)}$}}
\put(62,16){\mbox{$$}}  \put(72,16){\mbox{$\frac{th_r(1,t,t)}{h_r(1,1,t)}$}}
\put(45,6){\mbox{$$}}  \put(55,6){\mbox{$\frac{th_r(t,t,t)}{h_r(1,t,t)}$}}
\end{picture}
\vskip 0.2 cm

\noindent In the general case the de Casteljau algorithm of the \muntz space 
$\mathcal{E}_{(k)}(n)$ is given by :
\vskip 0.3 cm

\hskip 0 cm {\bf{Given}} $p_i^0 = p_i$, $i=0,...,n$

\hskip 0 cm {\bf{for}} $r=1:n$  {\bf{do}}

\hskip 1 cm {\bf{for}} $i=0:n-r$ {\bf{do}} 

$$
\hskip -1cm 
p_i^{r} = \left( 1-\frac{t h_k(1^i,t^r)}{h_k(1^{i+1},t^{r-1})}\right) p_i^{r-1} + 
\frac{t h_k(1^i,t^r)}{k_k(1^{i+1},t^{r-1})} p_{i+1}^{r-1}
$$

\hskip 1 cm {\bf{return}}

\hskip 0 cm {\bf{return}}

\hskip 0 cm $P(t) = p_0^n$.  

\begin{figure}
\hskip 2 cm
\includegraphics[height=6cm]{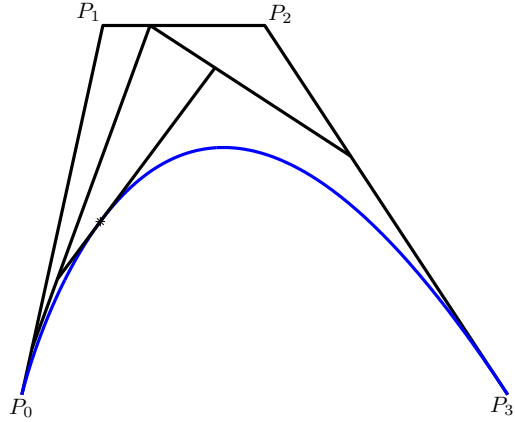}
\caption{The de Casteljau algorithm for the \muntz space 
$\mathcal{E}_{(2)}(3) = span(1,t^3,t^4,t^5)$ applied to the
Gelfond-B\'ezier curves associated with the control polygon
$(P_0,P_1,P_2,P_3)$ for the parameter $t = 1/2$.}
\label{fig:figure4}
\end{figure}
\begin{remark}\label{remarkshift}
Let $\lambda$ be a real partition associated with a \muntz space 
of order $n$, $\mathcal{E}_{\lambda}(n)$, and  let $P$ be an element of 
$\mathcal{E}_{\lambda}(n)$ written in the Gelfond-Bernstein basis as 
$$
P(t)  = \sum_{j=0}^{n} p_{j} \mathcal{H}^{n}_{j,\lambda}(t).  
$$
Denote by $q_i, i=0,...,n$ the control points of the function $P$
over an interval $[a,b]$ such that $0 <a < b <1$, namely 
$q_i = f_{P}(a^{n-i},b^i)$. Then from the properties of the blossom,
the function $P$ can also be written as 
$$
P(t)  = \sum_{j=0}^{n} q_{j} B^{n}_{j,\lambda}(t),  
$$
where $B^{n}_{j,\lambda},j=0,...,n$ is the Chebyshev-Bernstein basis of 
the space $\mathcal{E}_{\lambda}(n)$ over the interval $[a,b]$. Therefore, 
in some sense, the Gelfond-Bernstein basis over an interval contained in $[0,1]$
and does not contain the origin is exactly the Chebyshev-Bernstein basis. 
This has the drawback that if we reiterate the de Casteljau algorithm 
over intervals that does not contain the origin then we loose 
the simplifications in the algorithm that were brought up by the origin
through the splitting principle of Schur functions. To ovoid this drawback 
in practice, we should always make sure that the origin is a part 
of our interval. For example, to draw Gelfond-B\'ezier curves using the de 
Casteljau algorithm, we first subdivide the interval $[0,1]$ into the 
desired number of sub-intervals 
$[0=x_0,x_1],[x_1,x_2],....,[x_{m-1},x_m=1]$ and then apply successively 
the de Casteljau algorithm over the intervals $[0,x_s]\cup [x_s,x_{s+1}]$ 
for $s = m-1,m-2,...,1$.         
\end{remark} 

\section{The dimension elevation process}
Let $\Lambda_1 = (0=r_0,r_1,...,r_n)$ be a sequence of strictly increasing 
real numbers and let $H^{n}_{k,\Lambda_1}(t)$ be its corresponding 
Gelfond-Bernstein basis. Consider, now, a real number $\rho \neq r_i, 
i=0,...,n$. The \muntz space $E_{\Lambda_1}(n)$ is a subset of the  
\muntz space $E = span(1,t^{r_1},...,t^{r_n},t^{\rho})$. Therefore,
the Gelfond-Bernstein basis of the space $E_{\Lambda_1}(n)$ can be expressed
in terms of the Gelfond-Bernstein basis of the space $E$. Such expressions 
depend on the position of $\rho$ in the sequence $r_1 < r_2 < ...< r_n$ with respect 
to the increasing order. If we denote by $\Lambda_2$ the sequence obtained 
by arranging $(r_0=0,r_1,...,r_n,\rho)$ in a strictly increasing order, then we have  
\begin{proposition}\label{dimensionproposition}
If $\rho > r_n$, then for $k=0,...,n$, we have  
\begin{equation}\label{dimensionelevation1}
H^{n}_{k,\Lambda_1}(t) = \frac{\rho - r_{k}}{\rho} H^{n+1}_{k,\Lambda_2}(t)
+ \frac{r_{k+1}}{\rho} H^{n+1}_{k+1,\Lambda_2}(t).
\end{equation}
If $\rho < r_1$, then
\begin{equation*}
H^{n}_{0,\Lambda_1}(t) = H^{n+1}_{0,\Lambda_2}(t) + H^{n+1}_{1,\Lambda_2}(t)
\end{equation*}
and for $k\geq 1$, we have 
\begin{equation*}
H^{n}_{k,\Lambda_1}(t) = H^{n+1}_{k+1,\Lambda_2}(t).
\end{equation*}
If for a certain $s$, we have $r_s < \rho < r_{s+1}$, then 
for $k \geq s$, we have 
\begin{equation*}
H^{n}_{k,\Lambda_1}(t) = H^{n+1}_{k+1,\Lambda_2}(t).
\end{equation*}
\begin{equation*}
H^{n}_{s-1,\Lambda_1}(t) = \frac{\rho - r_{s-1}}{\rho} H^{n+1}_{s-1,\Lambda_2}(t)
+ H^{n+1}_{s,\Lambda_2}(t)
\end{equation*}
and for $k < s-1$
\begin{equation*}
H^{n}_{k,\Lambda_1}(t) = \frac{\rho - r_{k}}{\rho} H^{n+1}_{k,\Lambda_2}(t)
+ \frac{r_{k+1}}{\rho} H^{n+1}_{k+1,\Lambda_2}(t).
\end{equation*}
\end{proposition}
\begin{proof}
We will only prove (\ref{dimensionelevation1}), as the other cases
can be proven similarly.
From the definition of the Gelfond-Bernstein basis, the right hand 
side of equation (\ref{dimensionelevation1}) is given by (for $k \leq n-1$)
\begin{equation}\label{iterationequation1}
(-1)^{n-k} r_k r_{k+1}...r_{n} \left( [r_{k+1},...,\rho]f_t - 
(\rho - r_{k}) [r_{k},...,\rho]f_t \right).
\end{equation}
From the definition of the divided difference, we have 
\begin{equation*}
[r_{k+1},...,\rho]f_t - [r_{k},...,r_{n}]f_t = (\rho - r_{k})
[r_{k},...,\rho]f_t.
\end{equation*}
Inserting the last equation into (\ref{iterationequation1})
conclude the proof of the lemma for $k \leq n-1$.
For $k=n$, the left hand side of (\ref{dimensionelevation1}) 
is equal to 
\begin{equation*}
t^{\rho} - (\rho-r_{n})[r_n,\rho]f_t = t^{r_n} = 
H^{n}_{n,\Lambda_1}(t).  
\end{equation*}
\end{proof}
Consider now an element of $E_{\Lambda_1}(n)$, written in the 
Gelfond-Bernstein bases of the spaces $E_{\Lambda_1}(n)$ and 
$E_{\Lambda_2}(n+1)$ as  
\begin{equation}\label{expansion}
P(t) = \sum_{k=0}^{n} H_{k,\Lambda_1}^{n}(t) P_{k} =
\sum_{k=0}^{n+1} H_{k,\Lambda_2}^{n+1}(t) \tilde{P}_{k},
\end{equation}
where $\Lambda_1$ and $\Lambda_2$ refer to the sequences in the statement 
of the last proposition. Using proposition \ref{dimensionproposition} 
to detect the coefficients of $H_{k,\Lambda_2}^{n+1}(t)$
in the expansion (\ref{expansion}), we readily find 
\begin{corollary}\label{elevationtheorem2}
The Gelfond-B\'ezier points $\tilde{P_k}$ in (\ref{expansion})
are related to the Gelfond-B\'ezier points $P_{k}$ 
by the relations  
$$
\tilde{P}_0 = P_0, \quad \tilde{P}_{n+1} = P_{n},
$$
and if $\rho > r_n$ then for $k=1,2,...,n$
\begin{equation}\label{de}
\tilde{P}_k = \frac{r_{k}}{\rho} \; P_{k-1} + 
\frac{(\rho-r_{k})}{\rho} \; P_{k}.
\end{equation}
If $\rho < r_1$ then for $k=0,...,n-1$, we have 
\begin{equation*}
\tilde{P}_{k+1} = P_{k}, 
\end{equation*}
and if for an $s$, we have $r_s < \rho < r_{s+1}$, then 
for $k=1,2,...,s-1$
\begin{equation*}
\tilde{P}_k = \frac{r_{k}}{\rho} \; P_{k-1} + \frac{(\rho-r_{k})}{\rho} \; P_{k},
\end{equation*}
and for $k=s,...,n+1$
\begin{equation*}
\tilde{P}_k = P_{k-1}. 
\end{equation*}
\end{corollary}
Let $n$ be a fixed integer and let $(0=r_0,r_1,...,r_n,r_{n+1},....,r_{m},...)$
be an infinite sequence of strictly increasing real numbers. For any positive integer $q$, 
we denote by $\Lambda_q = (0=r_0,r_1,...,r_q)$. Let $P$ be an element of the \muntz 
space $E_{\Lambda_n}(n)$ written as 
\begin{equation}\label{expansion2}
P(t) = \sum_{k=0}^{n} H_{k,\Lambda_n}^{n}(t) P_{k} =
\sum_{k=0}^{m} H_{k,\Lambda_m}^{m}(t) \tilde{P}_{k}; \quad m > n.
\end{equation}
Then, from corollary \ref{elevationtheorem2}, equation (\ref{de}), the control 
points $\tilde{P}_{k}$ can be computed using the following
corner cutting scheme : For $i=0,1,...,n$, we set $P_{i}^{0} =P_{i}$ and for
$j=1,2,...m-n$, we construct iteratively new polygons 
$(P_0^{j},P_1^{j},....,P_{n+j}^j)$ using the inductive rule  
\begin{equation}\label{initial}
P_{0}^{j} =P_{0}^{j-1} \quad  P_{n+j}^{j} = P_{n+j-1}^{j-1}
\end{equation}
and for $i=1,...,n+j-1$
\begin{equation}\label{cornercutting} 
P_{i}^{j} = \frac{r_{i}}{r_{n+j}} P_{i-1}^{j-1} +
\left( 1-\frac{r_{i}}{r_{n+j}} \right) P_{i}^{j-1}.
\end{equation}
\begin{figure}[h!]
\includegraphics[height=3.7cm]{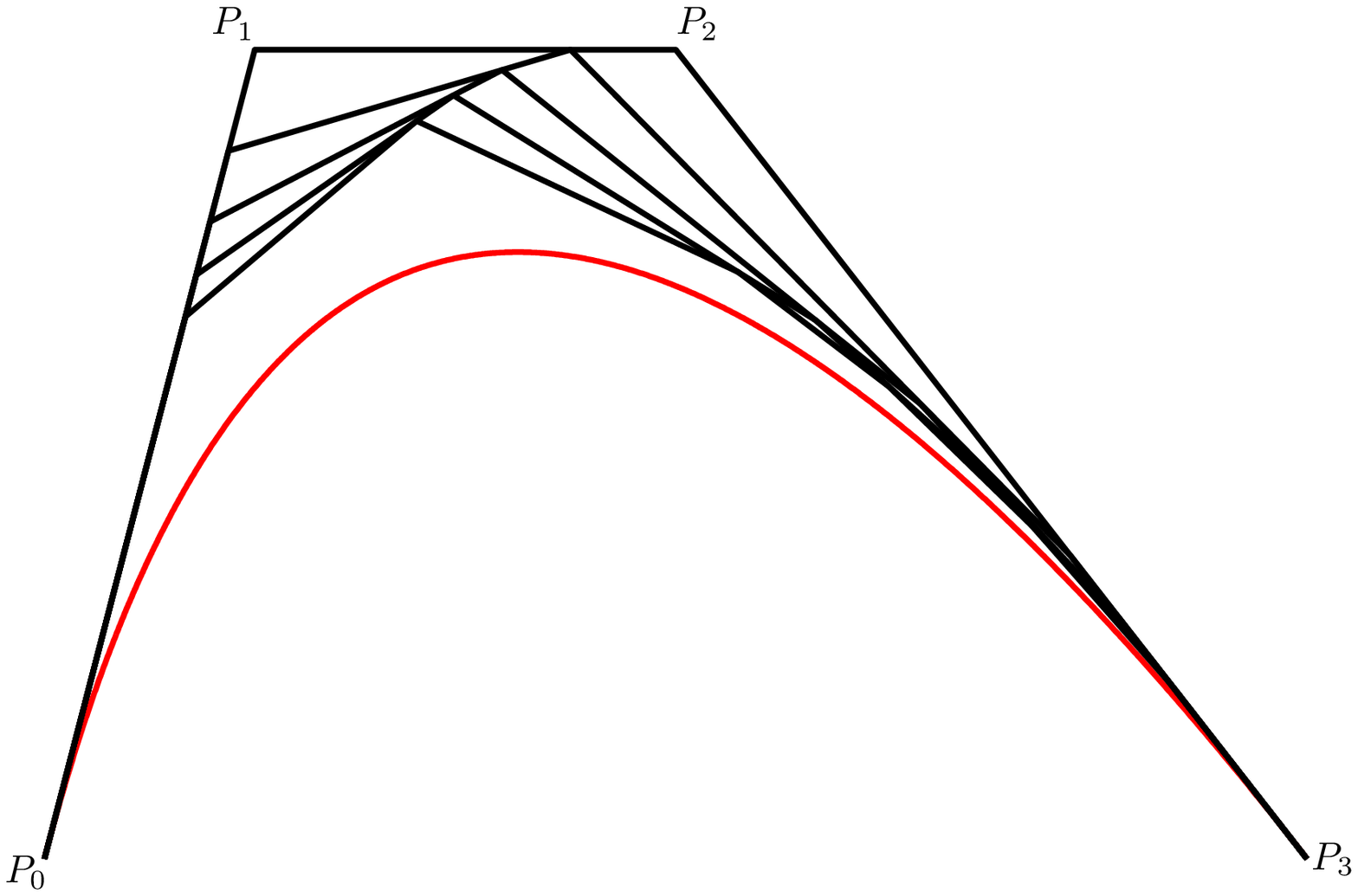}
\hskip 1.2 cm
\includegraphics[height=3.7cm]{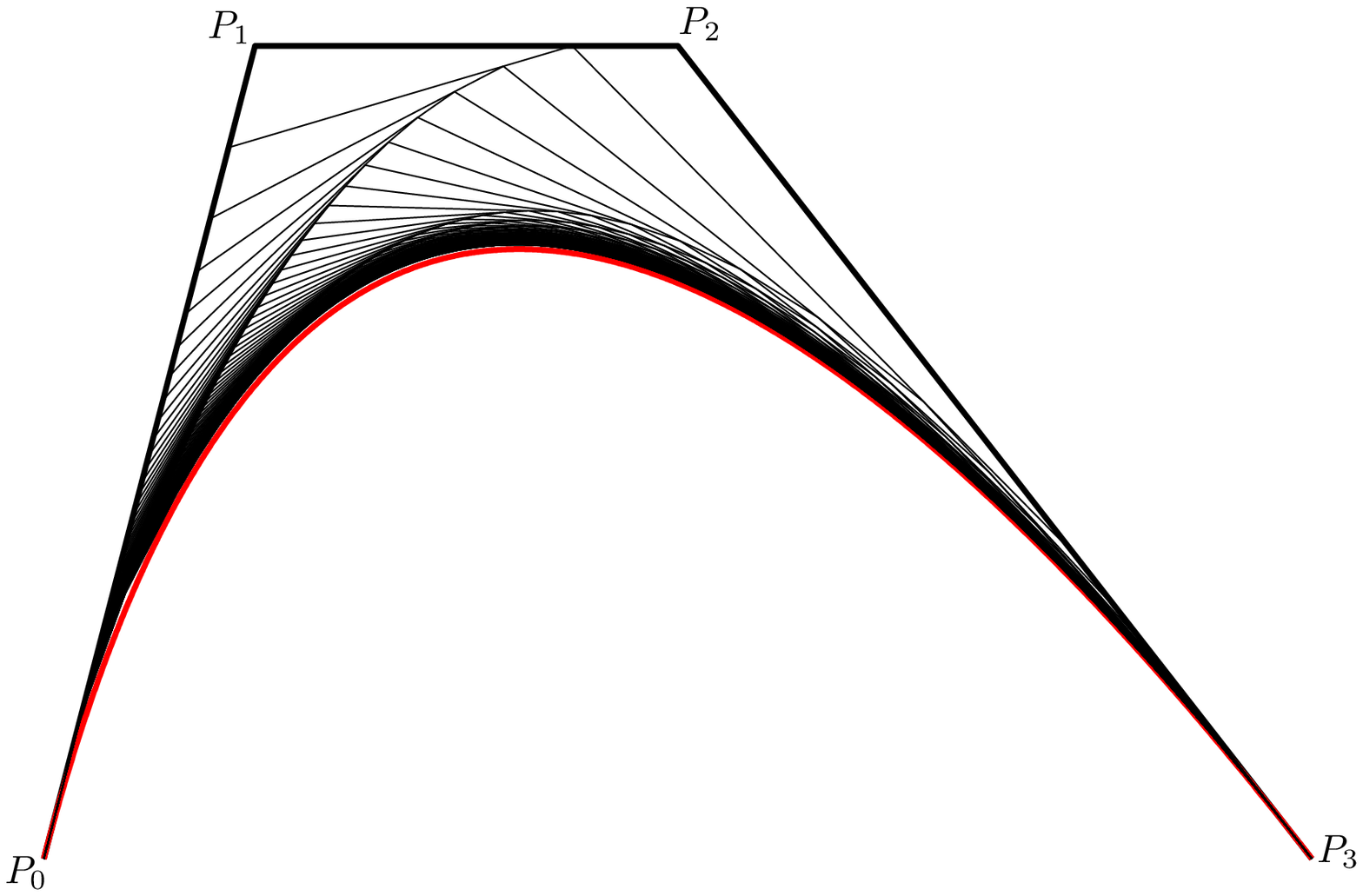}
\caption{The sequence of polygons generated by the corner 
cutting scheme (\ref{initial}) and (\ref{cornercutting})
and parameters $n=3, r_1=1, r_2=2, r_3 =3$ and $r_j = 2j$ for $j \geq 4$.
(left, four iterations of the scheme; right, 100 iterations of the scheme). 
The red curve is the B\'ezier curve associated with the control polygon 
$(P_0,P_1,P_2,P_3)$.}
\label{fig:figure5}
\end{figure}
In the case $r_i = i$ for any integer $i$, then we obtain 
the degree elevation algorithm, in which it is well know that 
the generated control polygon converges to the underlying
B\'ezier curve as $m$ goes to infinity \cite{kobbelt}.
Now consider the case in which $r_i = i$ for $i=1,...,n$ and
$r_i = 2i$ for $i > n$. Figure \ref{fig:figure5} (left) shows
the generated polygons from the scheme (\ref{initial}) and (\ref{cornercutting}) 
from four iterations, while Figure \ref{fig:figure5} 
(right) shows the generated polygons from 100 iterations. 
The figure suggests the convergence of the generated polygons 
to the B\'ezier curve with control points $(P_0,P_1,...P_n)$.
Consider, now, the case in which $r_i = i$ for $i=1,...,n$, 
while $r_i = i^2$ for $i >n$.
Figure \ref{fig:figure6} (left) shows the generated polygons 
from four iterations, while Figure \ref{fig:figure6} (right)
shows the obtained polygons after 100 iterations. It is clear
from the figure that the limiting polygon does not converge
to the B\'ezier curve with control points $(P_0,P_1,...,P_n)$.
Now, consider, for example, the limiting polygon of the corner 
cutting scheme (\ref{initial}) and (\ref{cornercutting})
for the case $n=3$ and in which $r_1=2, r_2 = 4, r_3 =5$ and 
$r_i = 2i$ for $i >3$.
\begin{figure}
\vskip 0.5 cm{
\includegraphics[height=3.7cm]{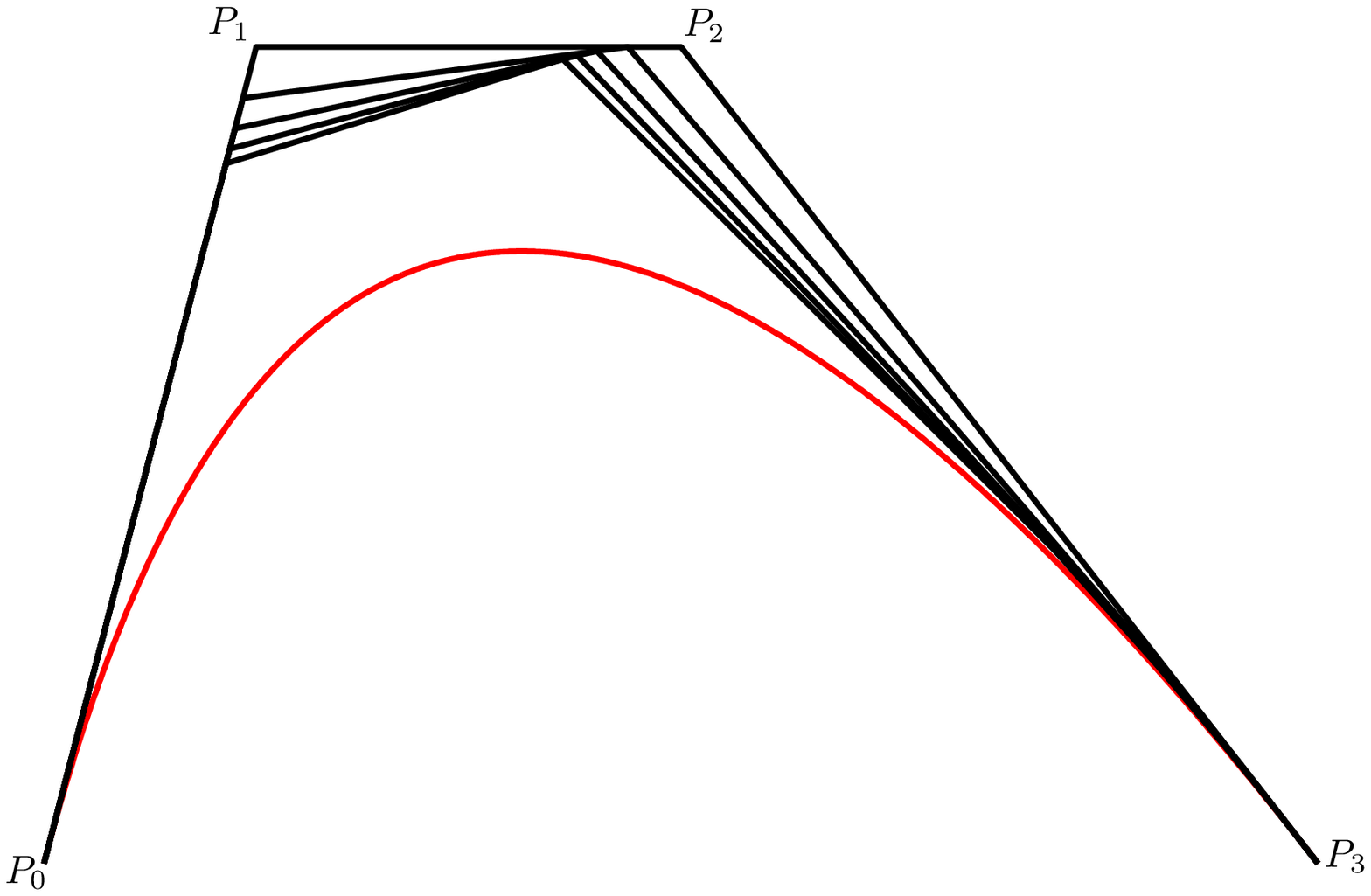}
\hskip 1.2 cm
\includegraphics[height=3.7cm]{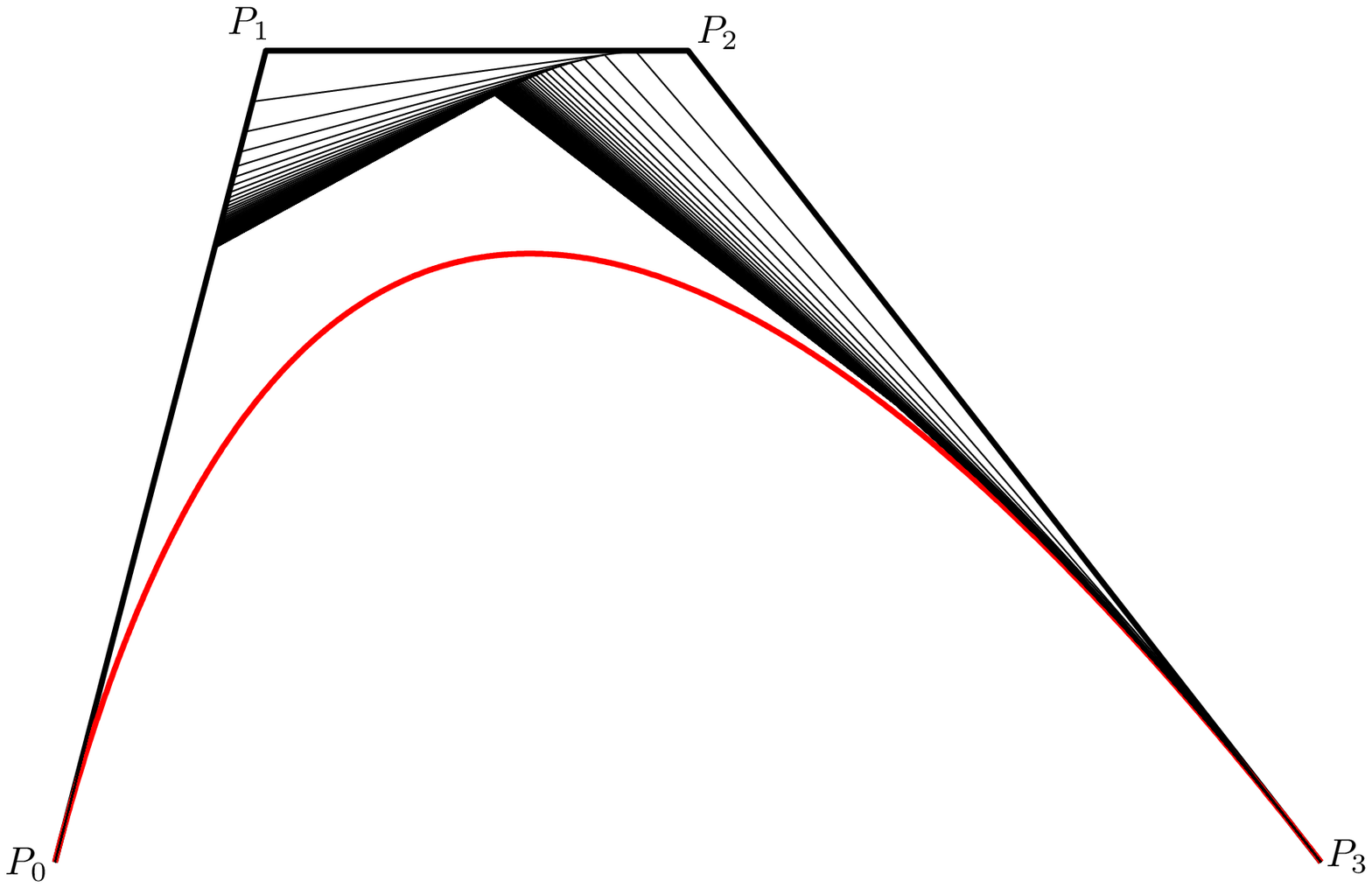}
\caption{The sequence of polygons generated by the corner 
cutting scheme (\ref{initial}) and (\ref{cornercutting})
and parameters $n=3, r_1=1, r_2=2, r_3 =3$ and $r_j = j^2$ for $j \geq 4$.
(left, four iterations of the scheme; right, 100 iterations of the scheme). 
The red curve is the B\'ezier curve associated with the control polygon 
$(P_0,P_1,P_2,P_3)$.}}
\label{fig:figure6}
\end{figure}
Figure \ref{fig:figure7} shows the generated polygons from 100 iterations and 
also shows the Gelfond-B\'ezier curve associated with the \muntz space
$F = span(1,t^{r_1},t^{r_2},t^{r_3}) = span(1,t^2,t^4,t^5)$ and 
control polygon $(P_0,P_1,P_2,P_3)$. The figure suggests that the limiting 
polygon converges to the Gelfond-B\'ezier curve. In fact, in \cite{aithaddou2},
the following was proven 

\begin{theorem}\label{muntzelevation}
Let $n$ be a fixed number and let 
$0 < r_1 < r_2 < ... r_n<r_{n+1}< ...<{r_m}<...$ be an 
infinite strictly increasing sequence of positive real numbers such that 
$\lim_{s\to\infty} r_{s} = \infty$. Then the limiting polygon
generated from a polygon $(P_0,P_1,...,P_n)$ in 
$\mathbb{R}^s, s \geq 1$ using the corner cutting scheme 
(\ref{initial}) and (\ref{cornercutting}) converges (pointwise 
and uniformly) to the Gelfond-B\'ezier curve associated with the \muntz space 
$span(1,t^{r_1},t^{r_2},...,t^{r_n})$ and control polygon 
$(P_0,P_1,...,P_{n})$ if and only if the real number $r_i$ satisfy
the condition 
\begin{equation}\label{muntzcondition}
\sum_{i=1}^{\infty} \frac{1}{r_i} = \infty
\end{equation}
\end{theorem}
The last theorem is a far reaching generalization of the statement that
the control polygons generated by the degree elevation algorithm converge 
to the underlying B\'ezier curve, namely, the latter is a consequence of 
the fact that 
$$
\sum_{n=1}^{\infty} \frac{1}{n} = \infty.
$$
Moreover, the emergence of the so-called \muntz condition (\ref{muntzcondition}) 
in Theorem \ref{muntzelevation} is rather surprising and raises the question
of a possible connections between the convergence of the polygons generated 
by the dimension elevation process of Gelfond-B\'ezier curves and 
the density questions in \muntz space. 
For a discussion on this matter we refer to our work in \cite{aithaddou2}.
\begin{figure}
\vskip 0.1 cm
\hskip 2cm
\includegraphics[height=4.5cm]{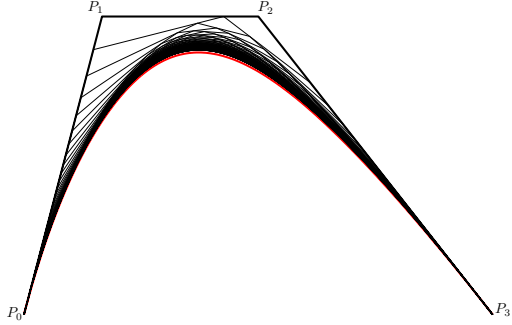}
\caption{The sequence of polygons generated from 100 iterations of 
the corner cutting scheme (\ref{initial}) and (\ref{cornercutting})
and parameters $n=3, r_1=2, r_2=4, r_3 =14$ and $r_j = 2j+10$ for $j \geq 4$. 
The red curve is the Gelfond-B\'ezier curve associated with the \muntz space 
$span(1,t^2,t^4,t^14)$ and control polygon $(P_0,P_1,P_2,P_3)$}
\label{fig:figure7}
\end{figure}
\section{Shifted Gelfond-B\'ezier curves and curve design}
As we have noted in remark \ref{remarkshift}, the Gelfond-Bernstein 
bases of \muntz spaces over an interval contained in $[0,1]$ and does not
contain the origin coincide, in some sense, with the Chebyshev-Bernstein bases.
Therefore, working with intervals that does not contain the origin has 
the drawback of loosing all the simplifications brought by the origin through 
the splitting principle of Schur functions. For curve design, in which 
for example we want to find conditions for the $C^k$ continuity between 
two Gelfond-B\'ezier curves, naturally one of the curves will be defined 
on an interval not containing the origin and then the  $C^k$ continuity 
conditions will be relatively complex as was shown in \cite{aithaddou1}.
One way to resolve this problem is to shift the origin to the left extremity of the interval
in which each of the two curves are defined.
This motivate the following definition.
\begin{definition}
Let $\Lambda = (0=r_{0},r_{1},...,r_{n})$ be a sequence of strictly
increasing real numbers and let $H^{n}_{k,\Lambda}, k=0,...,n$ 
be the Gelfond-Bernstein basis associated with the \muntz space $E_{\Lambda}(n)$
over the interval $[0,1]$. We define the shifted Gelfond-Bernstein basis
$\tilde{H}^{n}_{k,\Lambda}, k=0,...,n$ over an interval $[a,b]$ by 
\begin{equation*}
\tilde{H}^{n}_{k,\Lambda}(t) = H^{n}_{k,\Lambda}(\frac{t-a}{b-a});
\quad t\in [a,b]          
\end{equation*}
\end{definition}
Note that the shifted Gelfond-Bernstein basis $\tilde{H}^{n}_{k,\Lambda},
k=0,...,n$ over an interval $[a,b]$ is not a basis of the \muntz space
$E_{\Lambda}(n)$ but it is a basis of the shifted \muntz space 
$E_{\Lambda,a}(n) = span(1,(t-a)^{r_1},(t-a)^{r_2},...,(t-a)^{r_n})$.  
In the case the sequence $\Lambda = (0,1,...,n)$, then for any real number 
$a$, we have $E_{\Lambda}(n)=E_{\Lambda,a}(n)$, namely, the linear space of 
polynomials of degree $n$. In this case the shifted Gelfond-Bernstein basis 
over an interval $[a,b]$ coincide with the classical Bernstein basis over the 
interval $[a,b]$. 

All the relevant properties of shifted Gelfond-Bernstein bases over 
an interval $[a,b]$ can be deduced by simple manipulations from the 
non-shifted ones. For example, let $\Lambda_1 = (0=r_{0},r_{1},...,r_{n})$
and $\Lambda_2 = (0=s_{0},s_{1},...,s_{n})$ be two sequences of strictly
increasing real numbers and let $\tilde{H}^{n}_{k,\Lambda_1}, k=0,...,n$
be the shifted Gelfond-Bernstein basis over an interval $[a,b]$ associated 
with the sequence $\Lambda_1$ and $\tilde{H}^{n}_{k,\Lambda_2}, k=0,...,n$
be the shifted Gelfond-Bernstein basis over an interval $[b,c]$ associated 
with the sequence $\Lambda_2$. Consider now the following two shifted 
Gelfond-B\'ezier curves $\Gamma_1$ and $\Gamma_2$ with parameterizations 
\begin{equation*}
\begin{split}
& \Gamma_1: \quad P(t) = \sum_{k=0}^{n} \tilde{H}_{k,\Lambda_1}^{n}(t)
P_{k}; \quad t\in [a,b] \\
& \Gamma_2: \quad Q(t) = \sum_{k=0}^{n} \tilde{H}_{k,\Lambda_2}^{n}(t)
Q_{k}; \quad t\in [b,c].
\end{split} 
\end{equation*}
For simplicity, we assume that the real number $s_1$ in the sequence 
$\Lambda_2$ is equal to one. Then, in this case, from Theorem \ref{r1=1},
we have 
\begin{equation*}
P'(b) = \frac{r_n}{b-a} \Delta P_{n-1} 
\quad \textnormal{and} \quad 
Q'(c) = \frac{1}{c-b} \frac{\prod_{j=2}^{n} s_j}{\prod_{j=2}^{n} (s_j -1)} \Delta Q_0.
\end{equation*}
Therefore, a necessary and sufficient conditions for 
the two curves $\Gamma_1$ and $\Gamma_2$ to be $C^1$ at the point $P_n$ is that 
$$
P_{n} = Q_{0} 
\quad \textnormal{and} \quad
\frac{r_n}{b-a} \Delta P_{n-1} = \frac{1}{c-b} \frac{\prod_{j=2}^{n} s_j}
{\prod_{j=2}^{n} (s_j -1)} \Delta Q_0.
$$
Figure \ref{fig:figure8}, shows an example of $C_1$ continuity between two 
shifted Gelfond-B\'ezier curves of order $3$ associated respectively 
with the sequences $\Lambda_1 = (0=r_0,2,3,5)$ and $\Lambda_2 = (0=s_0,1,10,25)$
and defined respectively over the intervals $[1,2]$ and $[2,3]$.
It is possible to study the conditions for the $C^k, k \geq 2$ continuity and even 
define Gelfond splines. Such a study is still in progress and will be the subject
of a forthcoming contribution. 
\begin{figure}[h!]
\vskip 0.1 cm
\hskip 2cm
\includegraphics[height=7cm]{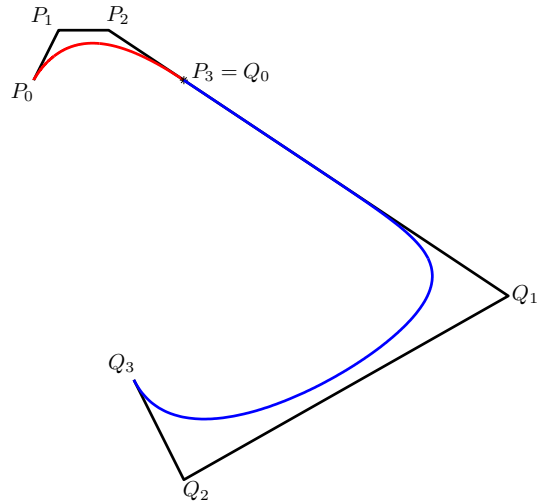}
\caption{$C^1$ continuity at the point $P_3$ between two shifted Gelfond-B\'ezier curves
associated with two different sequences. The shifted Gelfond-B\'ezier curve with control points
$(P_0,P_1,P_2,P_3)$ is associated with the sequence $\Lambda_1 = (0,2,3,5)$ and defined over 
the interval $[1,2]$, while the shifted Gelfond-B\'ezier curve with control points
$(Q_0,Q_1,Q_2,Q_3)$ is associated with the sequence $\Lambda_1 = (0,1,10,25)$ and 
defined over the interval $[1,2]$. (see text for more informations)}
\label{fig:figure8}
\end{figure}
\section{Conclusion}
In this work, we carried out a comprehensive study of the generalized 
Bernstein bases in \muntz spaces defined by Hirschman, Widder and Gelfond and 
that we termed here as Gelfond-Bernstein bases. We revealed their connection with 
the Chebyshev-Bernstein bases in \muntz spaces, thereby legitimating their role 
as a possible fundamental tool in computer aided geometric design concepts.
It it rather surprising that the Gelfond-Bernstein bases existed since 1949 and 
yet, to the best of our knowledge, they have never been incorporated into free form
curve design utilities. We hope that this work will motivate further study of the 
applications of Gelfond-B\'ezier curves and surfaces as well as Gelfond splines 
to computer aided geometric design.          

\subsubsection*{Acknowledgment : This work was partially supported by the MEXT 
Global COE project at Osaka University, Japan.} 

\vskip 0.2 cm

\label{references}

\end{document}